\documentclass[12pt]{amsart}
\usepackage{amssymb}
\usepackage{eucal}
\usepackage{amsmath}
\usepackage{amscd}
\usepackage[dvips]{color}
\usepackage{multicol}
\usepackage[all]{xy}
\usepackage{graphicx}
\usepackage{color}
\usepackage{colordvi}
\usepackage{xspace}
\usepackage{wasysym}
\usepackage[active]{srcltx}
\usepackage{url}

\usepackage{pdfsync}

\usepackage{tikz}
\usepackage[colorlinks,final,backref=page,hyperindex]{hyperref}

\begin{document}

\newcommand{\nc}{\newcommand}
\newcommand{\delete}[1]{}

\nc{\mlabel}[1]{\label{#1}}  
\nc{\mcite}[1]{\cite{#1}}  
\nc{\mref}[1]{\ref{#1}}  
\nc{\mbibitem}[1]{\bibitem{#1}} 

\delete{
\nc{\mlabel}[1]{\label{#1}  
{\hfill \hspace{1cm}{\small\tt{{\ }\hfill(#1)}}}}
\nc{\mcite}[1]{\cite{#1}{\small{\tt{{\ }(#1)}}}}  
\nc{\mref}[1]{\ref{#1}{{\tt{{\ }(#1)}}}}  
\nc{\mbibitem}[1]{\bibitem[\bf #1]{#1}} 
}

\nc{\redtext}[1]{\textcolor{red}{#1}}
\nc{\bluetext}[1]{\textcolor{blue}{#1}}
\nc{\greentext}[1]{\textcolor{green}{#1}}

\nc{\li}[1]{\textcolor{red}{Li: #1}}
\nc{\mg}[1]{\textcolor{blue}{M/G: #1}}

\newtheorem{theorem}{Theorem}[section]
\newtheorem{prop}[theorem]{Proposition}
\newtheorem{lemma}[theorem]{Lemma}
\newtheorem{coro}[theorem]{Corollary}
\newtheorem{prop-def}{Proposition-Definition}[section]
\newtheorem{claim}{Claim}[section]
\newtheorem{propprop}{Proposed Proposition}[section]
\newtheorem{conjecture}{Conjecture}
\newtheorem{assumption}{Assumption}
\newtheorem{condition}[theorem]{Assumption}
\newtheorem{question}[theorem]{Question}
\theoremstyle{definition}
\newtheorem{defn}[theorem]{Definition}
\newtheorem{exam}[theorem]{Example}
\newtheorem{remark}[theorem]{Remark}

\renewcommand{\emph}[1]{\textbf{#1}}
\renewcommand{\labelenumi}{{\rm(\alph{enumi})}}
\renewcommand{\theenumi}{\alph{enumi}}

\nc{\tred}[1]{\textcolor{red}{#1}}
\nc{\tblue}[1]{\textcolor{blue}{#1}}
\nc{\tgreen}[1]{\textcolor{green}{#1}}
\nc{\tpurple}[1]{\textcolor{purple}{#1}}
\nc{\btred}[1]{\textcolor{red}{\bf #1}}
\nc{\btblue}[1]{\textcolor{blue}{\bf #1}}
\nc{\btgreen}[1]{\textcolor{green}{\bf #1}}
\nc{\btpurple}[1]{\textcolor{purple}{\bf #1}}




\newcommand{\p}{P}        
\newcommand{\q}{Q}        
\newcommand{\D}{d}        
\newcommand{\pp}{\Pi}            
\newcommand{\DD}{D}             
\newcommand{\DDD}{\delta}    
\newcommand{\E}{E}                   
\newcommand{\J}{J}                    
\newcommand{\s}{S}                   
\newcommand{\T}{T}                   
\newcommand{\ct}{T}   
\newcommand{\ci}{J}    

\newcommand{\Cxto}{ \CC(x)_{T} } 
\newcommand{\Cxta}{ \CC(x)_{T,a}} 

\nc{\im}{\operatorname{im}}
\nc{\re}{\operatorname{Re}}

\nc{\const}{C}

\newcommand{\cum}{{\textstyle \varint}}
\providecommand{\abs}[1]{\lvert#1\rvert}

\def\reg{\textsuperscript{\textregistered}}
\def\tm{\leavevmode\hbox{$\rm {}^{TM}$}}
\newcommand{\mma}{Mathematica$\reg$}
\newcommand{\mpl}{Maple\tm}
\newcommand{\tma}{Theorema}


\nc{\idp}{I}    
\nc{\ID}{\mathbf{ID}}   
\nc{\CID}{\mathbf{CID}} 
\nc{\fid}[1]{\mathrm{ID}(#1)} 
\nc{\efid}[1]{\mathrm{ID}(#1)^*} 
\nc{\ee}{\vep}
\nc{\dop}[1]{{#1_d}}   
\nc{\pdop}[1]{(#1)_d}
\nc{\fop}[1]{{#1_f}}   
\nc{\pfop}[1]{(#1)_f}
\nc{\lbar}[1]{\overline{#1}}
\nc{\llbar}[1]{\overline{\overline{#1}}}

\nc{\adec}{\check{;}}
\nc{\dftimes}{\widetilde{\otimes}} \nc{\spr}{\cdot}
\nc{\disp}[1]{\displaystyle{#1}}
\nc{\bin}[2]{ (_{\stackrel{\scs{#1}}{\scs{#2}}})}  
\nc{\binc}[2]{ \left (\!\! \begin{array}{c} \scs{#1}\\
    \scs{#2} \end{array}\!\! \right )}  
\nc{\bincc}[2]{  \left ( {\scs{#1} \atop
    \vspace{-.5cm}\scs{#2}} \right )}  
\nc{\sarray}[2]{\begin{array}{c}#1 \vspace{.1cm}\\ \hline
    \vspace{-.35cm} \\ #2 \end{array}}
\nc{\bs}{\bar{S}} \nc{\dcup}{\stackrel{\bullet}{\cup}}
\nc{\dbigcup}{\stackrel{\bullet}{\bigcup}} \nc{\etree}{\big |}
\nc{\la}{\longrightarrow} \nc{\fe}{\'{e}} \nc{\rar}{\rightarrow}
\nc{\dar}{\downarrow} \nc{\dap}[1]{\downarrow
\rlap{$\scriptstyle{#1}$}} \nc{\uap}[1]{\uparrow
\rlap{$\scriptstyle{#1}$}} \nc{\defeq}{\stackrel{\rm def}{=}}
\nc{\diffa}[1]{\{#1\}}
\nc{\diffs}[1]{\Delta{#1}}
\nc{\sh}{Sh}
\nc{\dis}[1]{\displaystyle{#1}} \nc{\dotcup}{\,
\displaystyle{\bigcup^\bullet}\ } \nc{\sdotcup}{\tiny{
\displaystyle{\bigcup^\bullet}\ }} \nc{\hcm}{\ \hat{,}\ }
\nc{\hcirc}{\hat{\circ}} \nc{\hts}{\hat{\shpr}}
\nc{\lts}{\stackrel{\leftarrow}{\shpr}}
\nc{\rts}{\stackrel{\rightarrow}{\shpr}} \nc{\lleft}{[}
\nc{\lright}{]} \nc{\uni}[1]{\tilde{#1}} \nc{\wor}[1]{\check{#1}}
\nc{\free}[1]{\bar{#1}} \nc{\den}[1]{\check{#1}} \nc{\lrpa}{\wr}
\nc{\curlyl}{\left \{ \begin{array}{c} {} \\ {} \end{array}
    \right .  \!\!\!\!\!\!\!}
\nc{\curlyr}{ \!\!\!\!\!\!\!
    \left . \begin{array}{c} {} \\ {} \end{array}
    \right \} }
\nc{\leaf}{\ell}       
\nc{\longmid}{\left | \begin{array}{c} {} \\ {} \end{array}
    \right . \!\!\!\!\!\!\!}
\nc{\ot}{\otimes} \nc{\sot}{{\scriptstyle{\ot}}}
\nc{\otm}{\overline{\ot}}
\nc{\ora}[1]{\stackrel{#1}{\rar}}
\nc{\ola}[1]{\stackrel{#1}{\la}}
\nc{\scs}[1]{\scriptstyle{#1}} \nc{\mrm}[1]{{\rm #1}}
\nc{\margin}[1]{\marginpar{\rm #1}}   
\nc{\dirlim}{\displaystyle{\lim_{\longrightarrow}}\,}
\nc{\invlim}{\displaystyle{\lim_{\longleftarrow}}\,}
\nc{\mvp}{\vspace{0.5cm}} \nc{\svp}{\vspace{2cm}}
\nc{\vp}{\vspace{8cm}} \nc{\proofbegin}{\noindent{\bf Proof: }}
\nc{\proofend}{$\blacksquare$ \vspace{0.5cm}}
\nc{\sha}{{\mbox{\cyr X}}}  
\nc{\ssha}{\mathop{\mbox{\scyr X}}} 
\nc{\mxssha}{\mathop{\mbox{\scyr X}_\lambda}} 
\nc{\ncsha}{{\mbox{\cyr X}^{\mathrm NC}}} \nc{\ncshao}{{\mbox{\cyr
X}^{\mathrm NC,\,0}}}
\nc{\shpr}{\diamond}    
\nc{\shprm}{\overline{\diamond}}    
\nc{\shpro}{\diamond^0}    
\nc{\shprr}{\diamond^r}     
\nc{\shpra}{\overline{\diamond}^r}
\nc{\shpru}{\check{\diamond}} \nc{\catpr}{\diamond_l}
\nc{\rcatpr}{\diamond_r} \nc{\lapr}{\diamond_a}
\nc{\sqcupm}{\ot}
\nc{\lepr}{\diamond_e} \nc{\vep}{\varepsilon} \nc{\labs}{\mid\!}
\nc{\rabs}{\!\mid} \nc{\hsha}{\widehat{\sha}}
\nc{\lsha}{\stackrel{\leftarrow}{\sha}}
\nc{\rsha}{\stackrel{\rightarrow}{\sha}} \nc{\lc}{\lfloor}
\nc{\rc}{\rfloor} \nc{\sqmon}[1]{\langle #1\rangle}
\nc{\forest}{\calf} \nc{\ass}[1]{\alpha({#1})}
\nc{\altx}{\Lambda_X} \nc{\vecT}{\vec{T}} \nc{\onetree}{\bullet}
\nc{\Ao}{\check{A}}
\nc{\seta}{\underline{\Ao}}
\nc{\deltaa}{\overline{\delta}}
\nc{\trho}{\tilde{\rho}}
\nc{\tpow}[2]{{#2}^{\ot #1}}

\nc{\mmbox}[1]{\mbox{\ #1\ }} \nc{\ann}{\mrm{ann}}
\nc{\Aut}{\mrm{Aut}}
\nc{\bread}{\mrm{b}}
\nc{\can}{\mrm{can}} \nc{\colim}{\mrm{colim}}
\nc{\Cinf}{C^\infty} \nc{\rchar}{\mrm{char}}
\nc{\cok}{\mrm{coker}} \nc{\dtf}{{R-{\rm tf}}} \nc{\dtor}{{R-{\rm
tor}}}
\renewcommand{\det}{\mrm{det}}
\nc{\depth}{{\mrm d}}
\nc{\Div}{{\mrm Div}} \nc{\End}{\mrm{End}} \nc{\Ext}{\mrm{Ext}}
\nc{\Fil}{\mrm{Fil}} \nc{\Frob}{\mrm{Frob}} \nc{\Gal}{\mrm{Gal}}
\nc{\GL}{\mrm{GL}} \nc{\Hom}{\mrm{Hom}} \nc{\hsr}{\mrm{H}}
\nc{\hpol}{\mrm{HP}} \nc{\id}{\mrm{id}} 
\nc{\incl}{\mrm{incl}} \nc{\length}{\mrm{length}}
\nc{\LR}{\mrm{LR}} \nc{\mchar}{\rm char} \nc{\NC}{\mrm{NC}}
\nc{\mpart}{\mrm{part}} \nc{\pl}{\mrm{PL}}
\nc{\ql}{{\QQ_\ell}} \nc{\qp}{{\QQ_p}}
\nc{\rank}{\mrm{rank}} \nc{\rba}{\rm{RBA }} \nc{\rbas}{\rm{RBAs }}
\nc{\rbpl}{\mrm{RBPL}}
\nc{\rbw}{\rm{RBW }} \nc{\rbws}{\rm{RBWs }} \nc{\rcot}{\mrm{cot}}
\nc{\rest}{\rm{controlled}\xspace}
\nc{\rdef}{\mrm{def}} \nc{\rdiv}{{\rm div}} \nc{\rtf}{{\rm tf}}
\nc{\rtor}{{\rm tor}} \nc{\res}{\mrm{res}} \nc{\SL}{\mrm{SL}}
\nc{\Spec}{\mrm{Spec}} \nc{\tor}{\mrm{tor}} \nc{\Tr}{\mrm{Tr}}
\nc{\mtr}{\mrm{sk}}

\nc{\ab}{\mathbf{Ab}} \nc{\Alg}{\mathbf{Alg}}
\nc{\Algo}{\mathbf{Alg}^0} \nc{\Bax}{\mathbf{Bax}}
\nc{\Baxo}{\mathbf{Bax}^0} \nc{\Dif}{\mathbf{Dif}}
\nc{\CDif}{\mathbf{CDif}}
\nc{\CRB}{\mathbf{CRB}}
\nc{\CDRB}{\mathbf{CDRB}}
\nc{\RB}{\mathbf{RB}}
\nc{\DRB}{\mathbf{DRB}}
\nc{\RBo}{\mathbf{RB}^0} \nc{\BRB}{\mathbf{RB}}
\nc{\Dend}{\mathbf{DD}}
\nc{\Set}{\mathbf{Set}}
\nc{\bfk}{{\bf k}} \nc{\bfone}{{\bf 1}}
\nc{\base}[1]{{a_{#1}}} \nc{\detail}{\marginpar{\bf More detail}
    \noindent{\bf Need more detail!}
    \svp}
\nc{\Diff}{\mathbf{Diff}} \nc{\gap}{\marginpar{\bf
Incomplete}\noindent{\bf Incomplete!!}
    \svp}
\nc{\FMod}{\mathbf{FMod}} \nc{\mset}{\mathbf{MSet}}
\nc{\rb}{\mathrm{RB}} \nc{\Int}{\mathbf{Int}}
\nc{\Mon}{\mathbf{Mon}}
\nc{\remarks}{\noindent{\bf Remarks: }} \nc{\Rep}{\mathbf{Rep}}
\nc{\Rings}{\mathbf{Rings}} \nc{\Sets}{\mathbf{Sets}}
\nc{\DT}{\mathbf{DT}}

\nc{\BA}{{\mathbb A}} \nc{\CC}{{\mathbb C}}
\nc{\EE}{{\mathbb E}} \nc{\FF}{{\mathbb F}} \nc{\GG}{{\mathbb G}}
\nc{\HH}{{\mathbb H}} \nc{\LL}{{\mathbb L}} \nc{\NN}{{\mathbb N}}
\nc{\QQ}{{\mathbb Q}} \nc{\RR}{{\mathbb R}} \nc{\TT}{{\mathbb T}}
\nc{\VV}{{\mathbb V}} \nc{\ZZ}{{\mathbb Z}}


\nc{\cala}{{\mathcal A}} \nc{\calc}{{\mathcal C}}
\nc{\cald}{{\mathcal D}} \nc{\cale}{{\mathcal E}}
\nc{\calf}{{\mathcal F}} \nc{\calfr}{{{\mathcal F}^{\,r}}}
\nc{\calfo}{{\mathcal F}^0} \nc{\calfro}{{\mathcal F}^{\,r,0}}
\nc{\oF}{\overline{F}}  \nc{\calg}{{\mathcal G}}
\nc{\calh}{{\mathcal H}} \nc{\cali}{{\mathcal I}}
\nc{\calj}{{\mathcal J}} \nc{\call}{{\mathcal L}}
\nc{\calm}{{\mathcal M}} \nc{\caln}{{\mathcal N}}
\nc{\calo}{{\mathcal O}} \nc{\calp}{{\mathcal P}}
\nc{\calr}{{\mathcal R}} \nc{\calt}{{\mathcal T}}
\nc{\caltr}{{\mathcal T}^{\,r}}
\nc{\calu}{{\mathcal U}} \nc{\calv}{{\mathcal V}}
\nc{\calw}{{\mathcal W}} \nc{\calx}{{\mathcal X}}
\nc{\CA}{\mathcal{A}}

\nc{\fraka}{{\mathfrak a}} \nc{\frakB}{{\mathfrak B}}
\nc{\frakb}{{\mathfrak b}} \nc{\frakd}{{\mathfrak d}}
\nc{\oD}{\overline{D}}
\nc{\frakF}{{\mathfrak F}} \nc{\frakg}{{\mathfrak g}}
\nc{\frakm}{{\mathfrak m}} \nc{\frakM}{{\mathfrak M}}
\nc{\frakMo}{{\mathfrak M}^0} \nc{\frakp}{{\mathfrak p}}
\nc{\frakS}{{\mathfrak S}} \nc{\frakSo}{{\mathfrak S}^0}
\nc{\fraks}{{\mathfrak s}} \nc{\os}{\overline{\fraks}}
\nc{\frakT}{{\mathfrak T}}
\nc{\oT}{\overline{T}}
\nc{\frakX}{{\mathfrak X}} \nc{\frakXo}{{\mathfrak X}^0}
\nc{\frakx}{{\mathbf x}}
\nc{\frakTx}{\frakT}      
\nc{\frakTa}{\frakT^a}        
\nc{\frakTxo}{\frakTx^0}   
\nc{\caltao}{\calt^{a,0}}   
\nc{\oV}{\overline{V}}
\nc{\ox}{\overline{\frakx}} \nc{\fraky}{{\mathfrak y}}
\nc{\frakz}{{\mathfrak z}} \nc{\oX}{\overline{X}}
\nc{\oZ}{\overline{Z}}

\font\cyr=wncyr10
\font\scyr=wncyr8


\title{On Integro-Differential Algebras}
\author{Li Guo}
\address{Department of Mathematics and Computer Science,
         Rutgers University,
         Newark, NJ 07102}
\email{liguo@rutgers.edu}
\author{Georg Regensburger}
\address{
Johann Radon Institute for Computational
and Applied Mathematics (RICAM),
Austrian Academy of Sciences,
A-4040 Linz, Austria}
\email{georg.regensburger@oeaw.ac.at}
\author{Markus Rosenkranz}
\address{
 	School of Mathematics, Statistics and Actuarial Science,
	University of Kent,
	Canterbury CT2 7NF, England}
\email{M.Rosenkranz@kent.ac.uk}

\date{\today}


\begin{abstract}
  The concept of integro-differential algebra has been introduced
  recently in the study of boundary problems of differential
  equations.  We generalize this concept to that of
  integro-differential algebra with a weight, in analogy to the
  differential Rota-Baxter algebra. We construct free commutative
  integro-differential algebras with weight generated by a base
  differential algebra. This in particular gives an explicit
  construction of the integro-differential algebra on one generator. Properties of these free objects are
  studied.
\end{abstract}

\maketitle
\tableofcontents

\setcounter{section}{0}

\section{Introduction}
\mlabel{s:intro}

\subsection{Motivation and goal}

Differential algebra \cite{Kol,Ri} is the study of differentiation and nonlinear
differential equations by purely algebraic means, without using an underlying
topology. It has been largely successful in many important areas like:
uncoupling of nonlinear systems, classification of singular components, and
detection of hidden equations. There are various implementations that offer the
main algorithms needed for such tasks, for instance the
\texttt{DifferentialAlgebra} package in the \mpl\ system~\cite{BLOP}.

In view of applications, there is one crucial component that does not
fit well in differential algebra---the treatment of initial or
boundary conditions. The problem is that the elements of a
differential algebra or field are abstractions that cannot be
evaluated at a specific point. For bridging this gap (first in a
specific context of two-point boundary problems), a new framework was
set up in~\cite{R} with the following features:
\begin{itemize}
\item Differential algebras are enhanced by two evaluations
  (multiplicative functionals to the ground field) and two integral
  operators (Rota-Baxter operators), leading to the notion of analytic
  algebra.
\item The usual ring of differential operators is generalized to a
  ring of integro-differential operators.
\item Boundary problems are formulated in terms of the operator ring
  (differential equations as usual, boundary conditions in terms of the
  evaluations).
\item The Green's operator of a boundary problem is computed as an
  element of the operator ring.
\end{itemize}
The algebraic framework of boundary problems was subsequently refined and
extended by a multiplicative structure with results on the corresponding
factorizations along a given factorization of the differential
operator~\cite{RR,RRTB1}.  The factorization approach to boundary problems was
applied in~\cite{ACPRR, ACPaRR} to find closed-form and asymptotic expressions
for ruin probabilities and associated quantities in risk theory.

Moreover, it was realized that the algebraic theory of boundary problems is
intimately related to the theory of Rota-Baxter algebras, which can be regarded
as an algebraic study of both the integral and summation operators, even though
it originated from the probability study of G.~Baxter~\cite{Ba} in
1960. Rota-Baxter algebras have found extensive applications in mathematics and
physics, including quantum field theory and the classical Yang-Baxter
equation~\mcite{Bai,CK,EGK,Guwi,Gub,GZ}. In a nutshell, the relation with Rota-Baxter
algebras is this: In the differential algebra~$C^\infty(\RR)$, every point
evaluation~$\phi$ gives rise to a unique Rota-Baxter operator~$(1-\phi) \circ
\cum$, where~$\cum$ is any fixed integral operator, say~$f \mapsto \cum_0^x
f(\xi) \, d\xi$. See also Theorem~\ref{prop:char-intdiffalg} below for a
more general relation between evaluations and integral operators. We refer
to~\cite{Bav,Bav1} for an extensive study on algebraic properties of
integro-differential operators with polynomial coefficients and a single
evaluation (corresponding to initial value problems).

The algebraic approach to boundary problems is currently developed for
linear ordinary differential equations although some effort is under
way to cover certain classes of linear partial differential
equations~\cite{RRTB}. Various parts of the theory have been
implemented, first as external \mma-\tma\ reasoner~\cite{R}, then as
internal \tma\ code~\cite{RRTB,RRTB1}, and recently in a \mpl\ package with
new features for singular boundary problems~\cite{KRR}.

\subsection{Main results and outline of the paper}

Our main purpose in this paper is to construct free objects in the category of
$\lambda$-integro-differential algebras, which is the at the heart of the
algebraic framework of boundary problems described above. We use the
construction of free objects in a structure closely related to the
$\lambda$-integro-differential algebra, namely the differential Rota-Baxter
algebra. A Rota-Baxter algebra is an algebraic abstraction of a reformulation of
the integral by parts formula where only the integral operator appears. Free
commutative Rota-Baxter algebras were obtained in~\mcite{GK1,GK2} in terms of
shuffles and the more general mixable shuffles of tensor powers.

More recently the concept of a differential Rota-Baxter algebra was
introduced~\mcite{GK3} by putting a differential operator and a Rota-Baxter
operator of the same weight together such that one is the one side inverse of
the other as in the Fundamental Theorem of Calculus. One advantage of this
relatively independent combination of the two operators in a differential
Rota-Baxter algebra is that the free objects can be constructed quite easily by
building the free Rota-Baxter algebra on top of the free differential
algebra. Since the axiom of an integro-differential algebra requires more
intertwined relationship between the differential and Rota-Baxter operators, a
free integro-differential algebra is a quotient of a free differential
Rota-Baxter algebra. With this as the starting point of our construction of free
integro-differential algebras, our strategy is to find an explicitly defined
linear basis for this quotient from the known basis of the free differential
Rota-Baxter algebra by tensor powers.
For this purpose we use regular differential algebras as our basic building
block for the tensor powers.

In Section~\mref{sec:id}, we first introduce the concept of an integro-differential algebra of weight $\lambda$ and study their various characterizations, especially those in connection with differential Rota-Baxter algebras. In Section~\mref{sec:commfreeid}, we start with recalling free commutative Rota-Baxter algebras of weight $\lambda$ and then free commutative differential Rota-Baxter algebras of weight $\lambda$ and derive the existence of free commutative integro-differential algebras. The explicit
construction of free objects in the category of $\lambda$-integro-differential algebras is carried out in
Section~\mref{sec:fida} (Theorem~\mref{thm:intdiffa}) with a preparation on regular differential algebras and a detailed discussion on the regularity of the differential algebras of differential polynomials and rational functions.

\section{Integro-differential algebras of weight $\lambda$}
\mlabel{sec:id}

We first introduce the concepts and basic properties
related to $\lambda$-integro-differential algebras.

\subsection{Definitions and preliminary examples}

We recall the concepts of a derivation with weight, a Rota-Baxter operator with
weight and a differential Rota-Baxter algebra with weight, before introducing
our definition of an integro-differential algebra with weight.

\begin{defn}
  Let $\bfk$ be a unitary commutative ring. Let $\lambda\in \bfk$ be
  fixed.
  \begin{enumerate}
  \item A \emph{differential $\bfk$-algebra of weight $\lambda$} (also
    called a \emph{$\lambda$-differential $\bfk$-algebra}) is a unitary
    associative $\bfk$-algebra $R$ together with a linear operator
    $\D\colon R\to R$ such that
    \begin{equation}
      \D(xy)=\D(x) y+x \D(y)+ \lambda \D(x)\D(y)\quad \text{for all } x,y\in R, \mlabel{eq:diff}
    \end{equation}
    and
    \begin{equation}
      \D(1)=0.
      \mlabel{eq:diffc}
    \end{equation}
    Such an operator is called a \emph{derivation of weight $\lambda$} or a
    \emph{$\lambda$-derivation}.
  \item A \emph{Rota-Baxter $\bfk$-algebra of weight $\lambda$} is an
    associative $\bfk$-algebra $R$ together with a linear operator
    $\p\colon R\to R$ such that
    \begin{equation}
      \p(x)\p(y)=\p(x\p(y))+\p(\p(x)y)+ \lambda \p(xy)\quad \text{for
        all }  x,y\in R.
      \mlabel{eq:rba}
    \end{equation}
    Such an operator is called a \emph{Rota-Baxter operator of weight
      $\lambda$} or a \emph{$\lambda$-Rota-Baxter operator}.
  \item A \emph{differential Rota-Baxter k-algebra of weight
      $\lambda$} (also called a \emph{$\lambda$-differential
      Rota-Baxter $\bfk$-algebra}) is a differential $\bfk$-algebra
    $(R,\D)$ of weight $\lambda$ and a Rota-Baxter operator $\p$ of
    weight $\lambda$ such that
    \begin{equation}
      \D \circ \p = \id_R.
      \mlabel{eq:idcomprotabaxter}
    \end{equation}
  \item An \emph{integro-differential $\bfk$-algebra of weight
      $\lambda$} (also called a \emph{$\lambda$-integro-differential
      $\bfk$-algebra}) is a differential $\bfk$-algebra $(R,\DD)$ of
    weight $\lambda$ with a linear operator $\pp\colon R \to R$ such
    that
    \begin{equation}
      \DD \circ \pp = \id_R \mlabel{eq:idcomp}
    \end{equation}
    and
    \begin{equation}
      \pp(\DD(x)) \pp(\DD(y)) = \pp(\DD(x)) y + x \pp(\DD(y)) -
      \pp(\DD(xy))  \quad \text{for all } x,y\in R.
      \mlabel{eq:diffbaxter}
    \end{equation}
  \end{enumerate}
  \mlabel{def:main}
\end{defn}

When there is no danger of confusion, we will
suppress $\lambda$ and $\bfk$ from the notations.  We will also denote
the set of non-negative integers by $\NN$.

Note that we require that a derivation~$\D$ satisfies $\D(1)=0$. This follows from Eq.~(\mref{eq:diff}) automatically when $\lambda=0$, but is a non-trivial restriction when $\lambda\neq 0$.  In
the next section, we give equivalent characterizations of the
\emph{hybrid Rota-Baxter axiom}~(\mref{eq:diffbaxter}) and discuss
its relation to the \emph{Rota-Baxter axiom}~(\mref{eq:rba}) as
well as consequences of the \emph{section
  axiom}~\eqref{eq:idcomp}. Note that the hybrid Rota-Baxter axiom
does not contain a term with the weight $\lambda$.

We next give some simple examples of differential Rota-Baxter algebras
and integro-differential algebras. As we shall see below
(Lemma~\mref{lem:intdiffRotaBaxter}), the latter are a special
case of the former.
Further examples will be given in later
sections. In particular, the algebras of $\lambda$-Hurwitz series are
integro-differential algebras (Proposition~\ref{pp:hurwdrb}). By
Theorem~\mref{thm:intdiffa}, every regular differential algebra naturally
gives rise to the corresponding free integro-differential algebra.

\begin{exam}
  \begin{enumerate}
  \item By the First Fundamental Theorem of Calculus
    \begin{equation}
      \frac{d}{dx} \Big(\int_a^x f(t)dt\Big) = f(x)
      \mlabel{eq:cal}
    \end{equation}
and the conventional integration-by-parts formula
\begin{equation}
  \mlabel{eq:conv-intparts}
  \int_a^x f(t) g'(t) dt= f(t) g(t)-f(a)g(a)-\int_a^x f'(t) g(t) dt,
\end{equation}
$(\Cinf(\RR),d/dx, \int_a^x)$ is an integro-differential algebra of weight
$0$. As we shall see later in Theorem~\mref{pp:charintdiff},
integration by parts is in fact equivalent to the hybrid
Rota-Baxter axiom \eqref{eq:diffbaxter}.
\item\label{it:divdiff} The following example from~\mcite{GK3} of a differential
  Rota-Baxter algebra is also an integro-differential algebra.
    Let $\lambda \in \RR$, $\lambda \neq 0$.  Let $R=\Cinf(\RR)$ denote the
    $\RR$-algebra of smooth functions $f\colon \RR\to \RR$, and consider the
    usual ``difference quotient'' operator $\DD_\lambda$ on $R$ defined by
    \begin{equation}
      (\DD_\lambda(f))(x) = (f(x+\lambda) - f(x))/\lambda.
      \mlabel{eq:ldiff}
    \end{equation}
    Then $\DD_\lambda$ is a $\lambda$-derivation on $R$.  When
    $\lambda=1$, we obtain the usual difference operator on
    functions. Further, the usual derivation is $\disp{\DD_0:=
      \lim_{\lambda \to 0} \DD_\lambda.}$ Now let $R$ be an
    $\RR$-subalgebra of $\Cinf(\RR)$ that is closed under the
    operators
    \begin{equation*}
      \pp_0(f)(x)=-\int_x^\infty f(t)dt,\quad
      \pp_\lambda(f)(x)=-\lambda\sum_{n\geq 0} f(x+n\lambda).
    \end{equation*}
    For example, $R$ can be taken to be the $\RR$-subalgebra generated
    by $e^{-x}$: $R=\sum_{k\geq 1} \RR e^{-kx}$.
       Then $\pp_\lambda$ is
    a Rota-Baxter operator of weight $\lambda$ and, for the
    $\DD_\lambda$ in Eq.~(\mref{eq:ldiff}),
    \begin{equation*}
      \DD_\lambda\circ \pp_\lambda=\id_R  \quad \text{for all } x,y\in R, 0\neq \lambda\in \RR,
    \end{equation*}
    reducing to the fundamental theorem $\DD_0\circ \pp_0=\id_R$ when
    $\lambda$ goes to $0$.
    We note the close relations of~$(R, \DD_\lambda, \pp_\lambda)$ to the time
    scale calculus~\cite{ABOP} and the quantum calculus~\cite{KC}.
    \mlabel{it:lsum}

    The fact that $(R, \DD_\lambda, \pp_\lambda)$ is actually an
    integro-differential algebra follows from
    Theorem~\ref{prop:char-intdiffalg}(\ref{it:chepe}) since the kernel of
    $\DD_\lambda$ is just the constant functions (in the case $\lambda \ne 0$
    one uses that $R = \sum_{k\geq 1} \RR e^{-kx}$ does not contain periodic
    functions).
  \item\label{ex:counterexample} Here is one example of a differential Rota-Baxter algebra
that is not an integro-differential algebra~\cite[Ex.~3]{RR}. 
    Let $\bfk$ be a field of characteristic zero, $A = \bfk[y]/(y^4)$, and $(A[x], d)$, where
    $\D$ is the usual derivation with $\D(x^k)=k\,x^{k-1}$. We define a $\bfk$-linear map $\p$ on $A[x]$ by
  \begin{equation}
    \label{eq:baxter-not-diff}
    \p(f) = \pp (f) + f(0,0) \, y^2,
  \end{equation}
  where $\pp$ is the usual integral with $\pp(x^k)=
  x^{k+1}/(k+1)$. Since the second term vanishes under $\D$, we see
  immediately that $\D \circ \p =\id_{A[x]}$. For verifying the
  Rota-Baxter axiom~\eqref{eq:rba} with weight zero, we compute
  \begin{align*}
    & \p(f) \p(g) = \pp (f) \pp(g) +
    g(0,0) \, y^2 \, \pp (f)+ f(0,0) \, y^2  \pp(g) + f(0, 0) g(0, 0) \, y^4,\\
    & \p ( f \p (g)) = \pp (f \, (\pp (g) + g(0,0) \, y^2))
    = \pp (f \, \pp (g) ) + g(0,0)\, y^2 \, \pp (f), \\
    &\p (\p (f) g) = \pp (  (\pp(f) + f(0,0) \, y^2 ) \, g )
    = \pp (\pp(f) g )+ f(0,0) \, y^2 \, \pp(g).
  \end{align*}
    Since $y^4 \equiv 0$ and the usual integral $\pp$ fulfills the
  Rota-Baxter axiom~\eqref{eq:rba}, this implies immediately
  that $\p$ does also. However, it does not fulfill the hybrid Rota-Baxter~\eqref{eq:diffbaxter} since for example
  \[
  \p(\D(x))\p(\D(y))=\p(1)\p(0)=0
  \]
  but we obtain
  \[
  \p(\D(x)) y + x \p(\D(y)) - \p(\D(xy)) = \p(1) y + x \p(0) - \p(y) =  (x+y^2) y - x y = y^3.
  \]
  for the right-hand side.
  \end{enumerate}
  \mlabel{ex:diff}
\end{exam}

\subsection{Basic properties of integro-differential algebras with
  weight}
\label{ssec:intdiffalgs}
We first show that an integro-differential algebra with weight is a differential Rota-Baxter algebra of the same weight. We then give several equivalent conditions for integro-differential algebras.

\begin{lemma}
Let $(R,\DD)$ be a differential algebra of weight $\lambda$ with a linear operator $\pp:R\to R$ such that $\DD\circ \pp=\id_R$. Denote $\J=\pp\circ \DD$.
\begin{enumerate}
\item
The triple $(R,\DD,\pp)$ is a differential Rota-Baxter algebra of weight $\lambda$ if and only if
\begin{equation}
\pp(x)\pp(y)=\J(\pp(x)\pp(y))   \quad \text{for all } x,y\in R,
\mlabel{eq:rba_IJ}
\end{equation}
and if and only if
\begin{equation}
\J(x)\J(y)=\J(\J(x)\J(y))  \quad \text{for all } x,y\in R.
\mlabel{eq:rba_J}
\end{equation}
\mlabel{it:rbeq}
\item
Every integro-differential algebra is a differential Rota-Baxter algebra.
\mlabel{it:iddrb}
\end{enumerate}
\mlabel{lem:intdiffRotaBaxter}
\end{lemma}
Note that Eq.~\eqref{eq:rba_IJ} does not contain a term with $\lambda$.
Also note Eq.~\eqref{eq:rba_J} involves only the initialization $\J$ and shows in particular that $\im \J$ is a subalgebra.

\begin{proof}
(\mref{it:rbeq})
Using Eq.~\eqref{eq:diff}, we see that
\[
\DD(\pp(x)\pp(y))=x\pp(y)+\pp(x)y+\lambda x y.
\]
Hence the Rota-Baxter axiom
\begin{equation}
 \pp(x) \pp(y) = \pp(x\pp(y))+\pp(\pp(x)y) +\lambda \pp(xy)
 \mlabel{eq:rba_I}
\end{equation}
is equivalent to Eq.~\eqref{eq:rba_IJ}.
Moreover, substituting $\DD(x)$ for
$x$ and $\DD(y)$ for $y$ in Eq.~\eqref{eq:rba_IJ}, we get the equivalent identity \eqref{eq:rba_J}.
\smallskip

\noindent
(\mref{it:iddrb}) Since $\J\circ \pp = \pp \circ (\DD \circ \pp)=\pp \circ \id_R=\pp$,
we obtain Eq.~\eqref{eq:rba_IJ} from the hybrid Rota-Baxter
axiom~\eqref{eq:diffbaxter} by substituting $\pp(x)$ for $x$ and
$\pp(y)$ for $y$.
\end{proof}

We now give several equivalent conditions for an integro-differential algebra by
starting with a result on complementary projectors on algebras.

\begin{lemma}
  Let $\E$ and $\J$ be projectors on a unitary $\bfk$-algebra $R$
  such that $\E+\J = \id_{R}$. Then the following statements are
  equivalent:
  \begin{enumerate}
  \item $\E$ is an algebra homomorphism,
  \item $\J$ is a derivation of weight $-1$,
  \item $\ker \E = \im \J$ is an ideal and $\im \E
    = \ker \J$ is a unitary subalgebra.
  \end{enumerate}
\mlabel{lem:proj}
\end{lemma}
\begin{proof}
  ((a) $\Leftrightarrow$ (b))
It can be checked directly that
$\E(xy)=\E(x)\E(y)$ if and only if  $\J(xy)=\J(x)y+x\J(y)-\J(x)\J(y).$
Further it follows from $\E+\J=\id_R$ that
$\E(1)=1$ if and only if $\J(1)=0$.
\smallskip

\noindent
((a) $\Rightarrow$ (c)) is clear once we see that the assumption of the lemma implies $\ker E=\im J$ and $\im E=\ker J$.

\noindent
((c) $\Rightarrow$ (a)) Let $x, y \in R$. Since $R = \im \E \oplus
\ker \E$, we have $x=x_1+x_2$ and $y=y_1+y_2$ with
$x_1=\E(x),y_1=\E(y) \in \im \E$ and $x_2,y_2 \in \ker
\E$.  Then $\E(x_1y_1)= x_1y_1$ since $\im \E$ is by
assumption a subalgebra. Thus
 \[
 \E(xy)= \E(x_1y_1)+\E(x_1y_2)+\E(x_2 y_1) + \E(x_2
 y_2)=x_1y_1=\E(x)\E(y),
 \]
where the last three summands vanish assuming that  $\ker E$ is an ideal.
Moreover,  $1\in \im \E$ implies~$E(1) = 1$.
\end{proof}
We have the following characterizations of integro-differential algebras.
\begin{theorem}
  \label{prop:char-intdiffalg}
  Let $(R,\DD)$ be a differential algebra of weight $\lambda$ with a
  linear operator $\pp$ on $R$ such that $\DD \circ \pp =
  \id_R$.  Denote $\J=\pp\circ \DD$, called the \emph{initialization}, and $\E=\id_{R}-\J$, called the \emph{evaluation}. Then
  the following statements are equivalent:
  \begin{enumerate}
\item $(R,\DD,\pp)$ is an integro-differential algebra;
\mlabel{it:cha}
\item $\E(xy)=\E(x)\E(y)   \quad \text{for all } x,y\in R$;
\mlabel{it:chb}
\item $\ker E = \im \J$ is an ideal;
\mlabel{it:chc}
\item $\J(x \J(y))=x\J(y)$ and $\J(\J(x) y)=\J(x)y   \quad \text{for all } x,y\in R$;
\mlabel{it:chd}
\item $\J(x \pp(y))=x\pp(y)$ and $\J(\pp(x) y)=\pp(x)y   \quad \text{for all } x,y\in R$;
\mlabel{it:chdp}
\item
$x\pp(y) = \pp(\DD(x) \pp(y) ) + \pp(xy) + \lambda \pp(\DD(x) y )$ and
$  \pp(x)y =  \pp(\pp(x) \DD(y) ) + \pp(xy) + \lambda \pp(x \DD(y) )   \quad \text{for all } x,y\in R;$
\mlabel{it:intpart}
\item $(R,\DD,\pp)$ is a differential Rota-Baxter algebra and\\
    $\pp(\E(x)y) = \E(x) \pp(y)$ and $\pp(xE(y)) = \pp(x)\E(y)    \quad \text{for all } x,y\in R$;
\mlabel{it:chepe}
\item $(R,\DD,\pp)$ is a differential Rota-Baxter algebra and\\
    $\J(\E(x) \J(y))=\E(x)\J(y)$ and $\J( \J(x) \E(y))=\J(x)\E(y)   \quad \text{for all } x,y\in R.$
\mlabel{it:cheje}
\end{enumerate}
\mlabel{pp:charintdiff}
\end{theorem}

\begin{remark}
{\rm
\begin{enumerate}
\item
Items (\ref{it:chd}) and (\ref{it:chdp}) can be regarded as the invariance
formulation of the hybrid Rota-Baxter axiom.
\item
Item (\ref{it:intpart}) can be seen as a ``weighted'' noncommutative
version of integration by parts: One obtains it in case of weight zero by
substituting $\cum g$ for $g$ in the usual
formula~(\mref{eq:conv-intparts}). This motivates also the name
integro-differential algebra. Clearly, in the commutative case the respective
left and right versions are equivalent.
\item\mlabel{rem:linconst} Since $\im E= \ker \DD$, the identities in Items
  (\ref{it:chepe}) and (\ref{it:cheje}) can be interpreted as left/right linearity of
  respectively $\pp$ and $\J$ over the constants of the derivation $\DD$,
  restricted to~$\im \J$ in the case of~(\ref{it:cheje}). Note again that
  (\ref{it:chepe}) and (\ref{it:cheje}) do not contain a term with $\lambda$.
\end{enumerate}
}
\mlabel{rem:charintdiff}
\end{remark}

\begin{proof}
We first note that under the assumption, we have $\J^2=\pp\circ (\DD \circ \pp) \circ \DD=\pp\circ \id_R \circ \DD
=\J$ and so the initialization $\J$ and evaluation $\E$ are
projectors. Therefore
  \begin{equation}
    \ker \DD  = \ker \J  = \im \E
    \quad\text{and}\quad
    \im \pp  = \im \J  = \ker \E,
    \mlabel{eq:kerdimI}
  \end{equation}
and
\begin{equation*}
R=\ker \DD \oplus \im \pp
\end{equation*}
is a direct sum decomposition.

\noindent
((\ref{it:cha}) $\Leftrightarrow$ (\ref{it:chb})).
It follows from Lemma~\mref{lem:proj} since the hybrid Rota-Baxter axiom~\eqref{eq:diffbaxter}
can be rewritten as
 \begin{equation}
       \J(x) \J(y) = \J(x) y + x \J(y) -
       \J(xy)  \quad \text{for all } x,y\in R.
       \mlabel{eq:diffbaxter_J}
 \end{equation}

\noindent
((\ref{it:chb}) $\Leftrightarrow$ (\ref{it:chc})). It follows from Lemma~\mref{lem:proj}, since
$\ker \DD = \ker \J = \im \E$ is a unitary subalgebra
by Eq.~\eqref{eq:diff} and Eq.~\eqref{eq:diffc}.

\noindent ((\ref{it:cha}) $\Rightarrow$ (\ref{it:chdp})).
We obtain (\ref{it:chdp}) by substituting in Eq.~\eqref{eq:diffbaxter_J} respectively $\pp(y)$
for $y$ and $\pp(x)$ for $x$.

\noindent ((\ref{it:chdp}) $\Leftrightarrow$ (\ref{it:chd})).
Substituting respectively $\DD(y)$ for
$y$ and $\DD(x)$ for $x$ in (\ref{it:chdp}) gives (\ref{it:chd}). Conversely, substituting respectively $\pp(y)$ for
$y$ and $\pp(x)$ for $x$ in (\ref{it:chd}) gives (\ref{it:chdp}).

\noindent ((\ref{it:chdp}) $\Leftrightarrow$ (\ref{it:intpart})). It follows from Eq.~\eqref{eq:diff}.

\noindent ((\ref{it:cha}) $\Rightarrow$ (\ref{it:chepe})).  By
Lemma~\mref{lem:intdiffRotaBaxter}, $(R,\DD,\pp)$ is a differential Rota-Baxter
algebra. Furthermore, using Eq.~\eqref{eq:diff} and $\DD\circ \E =0$, we see
that
\[
\DD(\E(x)\pp(y))=E(x)y \quad \text{and} \quad \DD(\pp(x)\E(y))=x\E(y)
\]
and so
\begin{equation*}
\J(\E(x)\pp(y))=\pp(E(x)y) \quad \text{and} \quad \J(\pp(x)\E(y))=\pp(x\E(y)).
\end{equation*}
Since we have proved (\ref{it:chdp}) from (\ref{it:cha}), we can respectively
substitute $E(x)$ for $x$ and $E(y)$ for $y$ in (\ref{it:chdp}) to get
(\ref{it:chepe}).

\noindent ((\ref{it:chepe}) $\Leftrightarrow$ (\ref{it:cheje})).
Further, from $\pp(\E(x)y)=\E(x)\pp(y)$ we obtain
$$\J(\E(x)J(y))=\pp(\DD(\E(x)J(y)))=\pp(\E(x)\DD(y))=\E(x)\J(y),$$
Conversely, from $\J(\E(x)\J(y))=\E(x)\J(y)$ we obtain
$$\pp(\E(x)y)=\pp(\DD(\E(x)\pp(y))) = \J(\E(x)\pp(y)) = \J(\E(x)\J(\pp(y))) =
\E(x) \pp(y)$$ using $\pp=\J \circ \pp$ and $\DD(\E(x)\pp(y))=E(x)y$. This
proves the equivalence of the first equations in (\ref{it:chepe}) and
(\ref{it:cheje}); the same proof gives the equivalence of the second equations.

\noindent
((\ref{it:chd}) $\Rightarrow$ (\ref{it:chc})). This is clear since the
identities imply that $\im \J$ is an ideal.

\noindent
((\ref{it:cheje}) $\Rightarrow$ (\ref{it:chdp})). Note that $\J( \E(x) \J(y))=\E(x)\J(y)$  gives
\[
\J(x\J(y))-\J(\J(x) \J(y)) =x \J(y)-\J(x)\J(y)
\]
and hence $\J(x\J(y))= x \J(y)$ with the Rota-Baxter axiom in the
form of Eq.~\eqref{eq:rba_J}. The identity $\J(\J(x) y)=\J(x)y$ follows analogously.
\end{proof}

\section{Free commutative integro-differential algebras}
\mlabel{sec:commfreeid}

We first review the constructions of free commutative differential algebra
with weight, free commutative Rota-Baxter algebras and free commutative differential Rota-Baxter algebras. These constructions are then applied in Section~\ref{ss:fida} to obtain free commutative integro-differential algebras and will be applied in Section~\mref{sec:fida} to give an explicit construction of free commutative integro-differential algebras.

\subsection{Free and cofree differential algebras of weight $\lambda$}

We recall the construction~\mcite{GK3} of free commutative differential algebras of weight $\lambda$.

\begin{theorem}
  Let $X$ be a set. Let
  \begin{equation*}
    \diffs(X)=X\times \NN= \{ x^{(n)}\, \big|\, x\in X, n\geq 0\}.
  \end{equation*}
Let $\bfk\diffa{X}$ be the free commutative algebra
    $\bfk[\diffs{X}]$ on the set $\diffs{X}$.  Define $\D_X\colon
    \bfk\diffa{X} \to \bfk\diffa{X}$ as follows. Let $w=u_1\cdots u_k,
    u_i\in \diffs{X}$, $1\leq i\leq k$, be a commutative word from the
    alphabet set $\Delta(X)$.  If $k=1$, so that $w=x^{(n)}\in
    \Delta(X)$, define $\D_X(w)=x^{(n+1)}$. If $k>1$, recursively
    define
    \begin{equation}
      \D_X(w)=\D_X(u_1)u_2\cdots u_k+u_1\D_X(u_2\cdots u_k)+\lambda
      \D_X(u_1)\D_X(u_2\cdots u_k).
      \mlabel{eq:prodind}
    \end{equation}
    Further define $\D_X(1)=0$ and then extend $\D_X$ to
    $\bfk\diffa{X}$ by linearity.  Then $(\bfk\diffa{X},\D_X)$ is the
    free commutative differential algebra of weight $\lambda$ on the
    set $X$.
  \mlabel{thm:freediff}
\end{theorem}

  The use of $\bfk\diffa{X}$ for free commutative differential
  algebras of weight $\lambda$ is consistent with the notation of the
  usual free commutative differential algebra (when $\lambda=0$).

We also review the following construction from~\mcite{GK3}. For any commutative $\bfk$-algebra $A$, let $A^{\NN}$ denote the $\bfk$-module of all
functions $f\colon \NN \rightarrow A$.  We define the \emph{$\lambda$-Hurwitz product} on $A^{\NN}$
by defining, for any $f, g \in A^{\NN}$, $fg \in A^{\NN}$ by
\[(fg)(n) = \sum_{k=0}^{n}\sum_{j=0}^{n-k}\binc{n}{k}\binc{n-k}{j}
\lambda^{k}f(n-j)g(k+j).\]
We denote the $\bfk$-algebra
$A^{\NN}$ with this product by $DA$, and call it the $\bfk$-algebra of
\emph{$\lambda$-Hurwitz series over $A$}.
It was shown in~\mcite{GK3} that $DA$ is a differential Rota-Baxter
algebra of weight $\lambda$ with the operators
$$\DD\colon DA\to DA, \quad (\DD(f))(n)=f(n+1), n\geq 0, f\in DA,
$$
$$  \pp \colon DA\to DA, \quad (\pp(f))(n)=f(n-1), n\geq 1,
  (\pp(f))(0)=0, f\in DA.
$$
In fact, $DA$ is the cofree differential algebra of weight $\lambda$
on $A$.  We similarly have
\begin{prop}
  The triple $(DA, \DD, \pp)$ is an integro-differential algebra of
  weight $\lambda$.  \mlabel{pp:hurwdrb}
\end{prop}
\begin{proof}
  Since $(DA, \DD, \pp)$ is a differential Rota-Baxter algebra, we
  only need to show that $\pp(\E(x)y)=\E(x)\pp(y)$ for $x, y\in DA$ by Theorem~\mref{pp:charintdiff}. But this is
  clear since $\im \E=\ker\DD=A$ and $\pp$ is $A$-linear.
\end{proof}

\subsection{Free commutative Rota-Baxter algebras}

We briefly recall the construction of free commutative Rota-Baxter
algebras.  Let $A$ be a commutative
$\bfk$-algebra. Define
\begin{equation}
  \sha (A)= \bigoplus_{k\in\NN} A^{\otimes (k+1)} = A\oplus A^{\otimes
    2}\oplus \cdots,
\label{eq:freerb}
\end{equation}
where and hereafter all the tensor products are taken over $\bfk$ unless otherwise stated.
Let $\fraka=a_0\ot \cdots \ot a_m\in A^{\ot (m+1)}$ and $\frakb=b_0\ot
\cdots \ot b_n\in A^{\ot (n+1)}$. If $m=0$ or $n=0$, define
\begin{equation}
  \fraka \shpr \frakb =\left \{\begin{array}{ll}
      (a_0b_0)\ot b_1\ot \cdots \ot b_n, & m=0, n>0,\\
      (a_0b_0)\ot a_1\ot \cdots \ot a_m, & m>0, n=0,\\
      a_0b_0, & m=n=0.
    \end{array} \right .
    \mlabel{eq:shpr0}
\end{equation}
If $m>0$ and $n>0$, inductively (on $m+n$) define
\begin{eqnarray}
  \fraka \shpr \frakb & = &
  (a_0b_0)\ot \Big(
  (a_1\ot a_2\ot \cdots \ot a_m) \shpr (1_A \ot b_1\ot \cdots \ot b_n) \notag \\
  &&
  \qquad \qquad +
  \; (1_A \ot a_1\ot \cdots \ot a_m) \shpr (b_1\ot \cdots \ot b_n) \mlabel{eq:shpr}\\
  && \qquad \qquad +
  \lambda\, (a_1\ot \cdots \ot a_m) \shpr (b_1\ot \cdots \ot b_n)\Big).
  \notag
\end{eqnarray}
Extending by additivity, we obtain a $\bfk$-bilinear map
\begin{equation*}
  \shpr: \sha (A) \times \sha (A) \rar \sha (A).
\end{equation*}
Alternatively,
\begin{equation*}
  \fraka\shpr \frakb=(a_0b_0)\ot (\lbar{\fraka} \mxssha
  \lbar{\frakb}),
\end{equation*}
where~$\bar{\fraka} = a_1 \ot \cdots \ot a_m$, $\bar{\frakb} = b_1 \ot
\cdots \ot b_n$ and~$\mxssha$ is the mixable shuffle (quasi-shuffle)
product of weight $\lambda$~\mcite{Gub,GK1,Ho}, which specializes to the
shuffle product $\ssha$ when $\lambda=0$.

Define a $\bfk$-linear endomorphism $\p_A$ on $\sha (A)$ by assigning
\[
\p_A( a_0\otimes a_1\otimes \dots \otimes a_n) =1_A\otimes
a_0\otimes a_1\otimes \dots\otimes a_n, \] for all $a_0\otimes
a_1\otimes \dots\otimes a_n\in A^{\otimes (n+1)}$ and extending by
additivity.  Let $j_A\colon A\rar \sha (A)$ be the canonical inclusion
map.

\begin{theorem}\emph{$($\cite{GK1,GK2}$)$}
  \mlabel{thm:shua} The pair $(\sha(A),\p_A)$, together with the
  natural embedding $j_A\colon A\rightarrow \sha (A)$, is a free
  commutative Rota-Baxter $\bfk$-algebra on $A$ of weight $\lambda$.
  In other words, for any Rota-Baxter $\bfk$-algebra $(R,\p)$ and any
  $\bfk$-algebra map $\varphi\colon A\rar R$, there exists a unique
  Rota-Baxter $\bfk$-algebra homomorphism $\tilde{\varphi}\colon (\sha
  (A),\p_A)\rar (R,\p)$ such that $\varphi = \tilde{\varphi} \circ
  j_A$ as $\bfk$-algebra homomorphisms.
\end{theorem}

Since $\shpr$ is compatible with the multiplication in $A$, we will suppress the
symbol $\shpr$ and simply denote $x y$ for $x\shpr y$ in $\sha (A)$, unless
there is a danger of confusion.

Let $(A, \D)$ be a commutative differential $\bfk$-algebra of weight
$\lambda$.  Define an operator $\D_A$ on $\sha (A)$ by assigning
\begin{eqnarray}
  \lefteqn{ \D_A(a_0\otimes a_1\otimes \dots\otimes a_n)}\notag\\
  &=&
  \D(a_0)\otimes a_1\otimes \dots \otimes a_n + a_0a_1\otimes
  a_2 \otimes \dots \otimes a_n
  +\lambda \D(a_0) a_1\otimes a_2\otimes
  \dots \otimes a_n
\mlabel{eq:diffa}
\end{eqnarray}
for $a_0\otimes \dots \otimes a_n\in A^{\otimes (n+1)}$ and then
extending by $\bfk$-linearity. Here we use the convention that
$\D_A(a_0)=\D(a_0)$ when $n=0$.

\begin{theorem}\emph{$($\cite{GK3}$)$}
  \mlabel{thm:commdiffrb} Let $(A,\D)$ be a commutative differential
  $\bfk$-algebra of weight $\lambda$. Let
$j_A\colon A \rar \sha (A)$
    be the $\bfk$-algebra
    embedding $($in fact a morphism of differential $\bfk$-algebras of weight $\lambda$$)$.
The quadruple $(\sha(A), \D_A, \p_A, j_A)$ is a free
    commutative differential Rota-Baxter $\bfk$-algebra of weight
    $\lambda$ on $(A, \D)$.
  \label{thm:freecommdiffrb}
\end{theorem}

\subsection{The existence of free commutative integro-differential algebras}
\mlabel{ss:fida}

The free objects in the category of commutative integro-differential algebras of weight $\lambda$ are defined in a
similar fashion as for the category of commutative differential Rota-Baxter algebras.

\begin{defn}
  \label{def:freeintdiffalg}
  Let $(A,\D)$ be a $\lambda$-differential algebra over~$\bfk$.  A
  \emph{free integro-differential algebra of weight $\lambda$ on
    $A$} is an integro-differential algebra $(\fid{A}, \DD_A, \pp_A)$
  of weight $\lambda$ together with a differential algebra
  homomorphism $i_A\colon (A,\D)\to (\fid{A},\D_A)$ such that, for any
  integro-differential algebra $(R, \DD, \pp)$ of weight $\lambda$ and a
  differential algebra homomorphism $f\colon (A,\D)\to (R,\DD)$, there is a
  unique integro-differential algebra homomorphism $\free{f}\colon
  \fid{A}\to R$ such that $\free{f}\circ i_A= f.$ When $A=\bfk\diffa{u}$,
  we let $\fid{u}$ denote $\fid{A}$.  \mlabel{de:freeid}
\end{defn}

As in Theorem~\mref{thm:commdiffrb}, let $(\sha (A), \D_A, \p_A)$ be
the free commutative differential Rota-Baxter algebra generated by the
differential algebra $(A, \D)$. Then by
Theorem~\mref{pp:charintdiff}, we have

\begin{theorem}
  \label{thm:freex}
  Let $(A,\D)$ be a commutative differential
  $\bfk$-algebra of weight $\lambda$. Let $I_\ID$ be the differential Rota-Baxter ideal of $\sha(A)$
  generated by the set
  \begin{equation*}
    \{\J\big(\E(x) \, \J(y)\big)-\E(x) \, \J(y) \ \big|\ x,y\in
    \sha (A)\},
  \end{equation*}
  where~$\J$ and~$\E$ denote the projectors~$\p_A \circ \D_A$ and~$\id_A
  - \p_A \circ \D_A$, respectively. Let $\DDD_A$ (resp. $\pp_A$) denote
  $\D_A$ (resp. $\p_A$) modulo $I_\ID$.  Then the quotient
  differential Rota-Baxter algebra $(\sha(A)/I_\ID, \DDD_A,\pp_A)$,
  together with the natural map $i_A\colon A\to \sha(a) \to
  \sha(A)/I_\ID$, is the free integro-differential algebra of weight
  $\lambda$ on $A$.
    \mlabel{thm:freeid}
\end{theorem}
\begin{proof}
  Let a $\lambda$-integro-differential algebra $(R, \DD, \pp)$ be
  given. Then by Theorem~\mref{pp:charintdiff}, $(R, \DD, \pp)$ is
  also a $\lambda$-differential Rota-Baxter algebra. Thus by
  Theorem~\mref{thm:commdiffrb}, there is a unique homomorphism $
  \tilde{f}\colon \sha(A)\to R $ such that the left triangle of the
  following diagram commutes.
  \begin{equation*}
    \xymatrix{
      & (\sha(A), \D_A, \p_A)\ar^{\pi}[rd] \ar^{\tilde{f}}[dd] & \\
      (A, \D) \ar^{j_A}[ur] \ar^{f}[dr] &&
      (\sha(A)/I_\ID, \DDD_A,\pp_A) \ar_{\free{f}}[dl] \\
      & (R, \DD, \pp) & }
  \end{equation*}
  Since $(R, \DD, \pp)$ is a $\lambda$-integro-differential algebra,
  $\tilde{f}$ factors through $\sha(A)/I_\ID$ and induces the
  $\lambda$-integro-differential algebra homomorphism $\free{f}$ such
  that the right triangle commutes. Since $i_A = \pi \circ j_A$, we
  have $\free{f}\circ i_A=f$ as needed.

  Suppose $\free{f}_1: \sha(A)/I_\ID \to R$ is also a
  $\lambda$-integro-differential algebra homomorphism such that
  $\free{f}_1\circ i_A=f$. Define $\tilde{f}_1=\free{f}_1\circ \pi$.
  Then $\tilde{f}_1\circ j_A=f$. Thus by the universal property of
  $\sha(A)$, we have $\tilde{f}_1=\tilde{f}$. Since $\pi$ is
  surjective, we must have $\free{f}_1=\free{f}.$ This completes the
  proof.
\end{proof}

\section{Construction of free commutative integro-differential algebras}
\mlabel{sec:fida} As mentioned in Section~\ref{s:intro}, in integro-differential
algebras the relation between~$\D$ and~$\pp$ is more intimate than in
differential Rota-Baxter algebras. This makes the construction of their free
objects more complex. Having ensured their existence in
(Section~\ref{ss:fida}), we introduce a vast class of differential algebras
for which our construction applies (Section~\ref{ss:reg-diffalg}). Next we
present the details of the construction and some basic properties
(Section~\ref{ss:freec}), leading on to the proof that it yields the desired
free object (Section~\ref{ss:intdifpf}). The construction applies in
particular to rings of differential polynomials~$\bfk \diffa{u}$, yielding the
free object over one generator, and to the ring of rational functions
(Section~\ref{ss:example}).

\subsection{Regular differential algebras}
\mlabel{ss:reg-diffalg}

A free commutative integro-differential algebra can be regarded as a universal
way of constructing an integro-differential algebra from a differential
algebra. The easiest way of obtaining an integro-differential algebra from a
differential algebra occurs when~$(A, \D)$ already has an integral
operator~$\pp$. This means in particular that~$\D \circ \pp = \id_A$ so that the
derivation~$\D$ must be surjective. But often this will not be the case, for
example when~$A = \bfk\diffa{u}$ is the ring of differential polynomials
(where~$u$ is clearly not in the image of~$\D$). But even if we cannot define an
antiderivative (meaning a right inverse for~$\D$) on all of~$A$, we may still be
able to define one on~$\D(A)$ using an appropriate
\emph{quasi-antiderivative}~$\q$. This means we require~$\D(\q(y)) = y$ for~$y
\in \D(A)$ or equivalently $\D(\q(\D(x)))=\D(x)$ for $x\in A$. For a general
operator~$\D$, an operator~$\q$ with this property is called an inner inverse
of~$\D$. It exists for many important differential algebras, in particular for
differential polynomials (Proposition~\mref{pp:diffpoly}) and rational function
(Proposition~\mref{pp:rational}).

Before coming back to differential algebras, we recall some properties
of generalized inverses for linear maps on $\bfk$-modules; for further
details and references see~\cite[Section 8.1.]{NV}.

\begin{defn}
Let $L\colon M \to N$ be a linear map between $\bfk$-modules. \begin{enumerate}
\item
If a linear map $\bar{L}\colon N\to M$ satisfies $L\circ \bar{L}\circ L=L$, then $\bar{L}$ is called an \emph{inner inverse} of~$L$.
\item
If $L$ has an inner inverse, then $L$ is called \emph{regular}.
\item
If a linear map $\bar{L}\colon N\to M$ satisfies $\bar{L}\circ L \circ \bar{L}=\bar{L}$, then $\bar{L}$ is called an \emph{outer inverse} of $L$.
\item If $\bar{L}$ is an inner inverse and outer inverse of $L$, then $\bar{L}$
  is called a \emph{quasi-inverse} or \emph{generalized inverse} of $L$.
\end{enumerate}
\end{defn}

\begin{prop}
  Let $L\colon M \to N$ be a linear map between~$\bfk$-modules.
\begin{enumerate}
\item
If~$L$ has an inner inverse~$\bar{L}\colon N \to M$, then~$\s = L \circ
  \bar{L}\colon N \to N$ is a projector onto~$\im{L}$ and~$\E = \id_M -
  \bar{L} \circ L\colon M \to M$ is a projector onto~$\ker{L}$.
\label{it:comp}
\item
Given projectors $\s\colon N\to N$ onto $\im L$ and $\E\colon M\to M$ onto $\ker L$, there is a unique quasi-inverse $\bar{L}$ of $L$ such that $\im \bar{L}=\ker \E$ and $\ker \bar{L}=\ker \s$.
Thus a regular map has a quasi-inverse.
\label{it:qi}
\end{enumerate}
\label{pp:reg}
\end{prop}
\begin{proof}
(\mref{it:comp}) This statement is immediate.
\smallskip

\noindent
(\mref{it:qi})
If $L$ is regular, then by Item~(\mref{it:comp}), there are submodules $\ker \E\subseteq M$ and $\ker \s\subseteq N$ such that
$$M=\ker L \oplus \ker \E, \quad N=\im L \oplus \ker \s.$$
Thus $L$ induces a bijection $L\colon \ker \E \to \im L$. Define $\bar{L}\colon
N\to M$ to be the inverse of this bijection on $\im L$ and to be zero on $\ker
\s$, then we check directly that $\bar{L}$ is a quasi-inverse of $L$ and the
unique one such that $\im \bar{L}=\ker \E$ and $\ker \bar{L}=\ker \s$. See
also~\cite[Theorem 8.1.]{NV}.
\end{proof}

For a quasi-inverse~$\bar{L}$ of~$L$
we note the direct sums
\begin{equation*}
  M= \im
  {\bar{L}} \oplus \ker {L} \quad\text{and} \quad N= \im {L} \oplus \ker {\bar{L}} .
\end{equation*}
Moreover, let
\[
\J = \id_M - \E \quad \text{and} \quad \T = \id_N - \s,
\]
then we have the relations
\begin{alignat*}{3}
  M_E&:=\im{\E}= \ker {L}=\ker J, \quad& &
  M_J:=\im{\J}= \im {\bar{L}} =\ker E \\
  N_S&:=\im{\s}= \im{L} =\ker T, \quad & &
  N_T:=\im{\T}= \ker {\bar{L}}=\ker S.
\end{alignat*}
for the corresponding projectors.

The intuitive roles of the projectors~$\E$ and~$\J$ are similar as in
Section~\ref{ssec:intdiffalgs}, except that the ``evaluation''~$\E$ is not
necessarily multiplicative and the image of the ``initialization''~$\J$ need not
be an ideal. The projector~$\s$ may be understood as extracting the solvable
part of~$N$, in the sense of solving~$L(x) = y$ for~$x$, as much as possible for
a given $y \in
N$.

Let us elaborate on this. Writing respectively $y_\s = \s(y)$ and $y_\T = \T(y)$
for the ``solvable'' and ``transcendental'' part of $y$, the equation $L(x) =
y_\s$ is clearly solved by $x^* = \bar{L} (y_\s)$ while $L(x) = y_\T$ is only
solvable in the trivial case $y_\T = 0$. So the identity $L(x^*) = y - \T(y)$
may be understood in the sense that $x^*$ solves $L(x) = y$ except for the
transcendental part. We illustrate this in the following example.

\begin{exam}
  \mlabel{ex:ratfun} Consider the field~$\CC(x)$ of \emph{complex rational
    functions} with its usual derivation~$\D$. We take $\D$ to be the linear map
  $L\colon M\to N$ where $M=N=\CC(x)$. Any rational function can be represented
  by~$f/g$ with a monic denominator~$g = (x-\alpha_1)^{n_1} \cdots
  (x-\alpha_k)^{n_k}$ having distinct roots~$\alpha_i \in \CC$.  By partial
  fraction decomposition, it can be written uniquely as
  \begin{equation}
    r + \sum_{i=1}^k \sum_{j=1}^{n_i}
    \frac{\gamma_{ij}}{(x-\alpha_i)^j},
    \mlabel{eq:pfdecomp}
  \end{equation}
  where $r\in \CC[x]$ and $\gamma_{ij} \in \CC$.  Then for the domain $\CC(x)$
  of $d$, we have the decomposition
$$\CC(x)  = \ker d \oplus \CC(x)_J$$
with $\ker d = \CC$ and
$$\CC(x)_J = \left \{     r + \sum_{i=1}^k \sum_{j=1}^{n_i}
  \frac{\gamma_{ij}}{(x-\alpha_i)^j}\,\Big|\, r\in x \, \CC[x], \alpha_i\in \CC
  \text{ distinct}, \gamma_{ij}\in \CC\right\}$$ as the initialized space. For
the range $\CC(x)$ of $d$, we have the decomposition
$$\CC(x)= \im d \oplus \CC(x)_T, $$
with
$$\im d = \left \{ r + \sum_{i=1}^k \sum_{j=2}^{n_i}
    \frac{\gamma_{ij}}{(x-\alpha_i)^j}\,\Big|\, r\in \CC[x], \alpha_i\in \CC \text{ distinct}, \gamma_{ij}\in \CC\right\}$$
and
$$ \CC(x)_T =\left \{ \sum_{i=1}^k \frac{\gamma_i}{x-\alpha_i}\,\Big|\,
  \alpha_i\in \CC \text{ distinct}, \gamma_i\in \CC\right\}$$
as the transcendental space.

By Proposition~\mref{pp:reg} there exists a unique quasi-inverse $Q\colon
\CC(x)\to \CC(x)$ of $\D$ corresponding to the above decompositions, which we
can describe explicitly.  On~$\im d$ we define $Q$ by setting~$Q(x^k) =
x^{k+1}/(k+1)$ for~$k \ge 0$ and~$Q(1/(x-\alpha)^j) = -j/(x-\alpha)^{j-1}$ for
$j>1$, and we extend it by zero on~$\CC(x)_T$. Analytically speaking, the
quasi-antiderivative $Q$ acts as~$\cum_0^x$ on the polynomials and as
$\cum_{-\infty}^x$ on the solvable rational functions: Since~$\CC(x)$ is not an
integro-differential algebra, it is not possible to use a single integral
operator. The associated codomain projector~$\s = d \circ \q$ extracts the
solvable part by filtering out the residues $1/(x-\alpha)$; their
antiderivatives would need logarithms, which are not available in~$\CC(x)$. The
domain projector~$\E = \id_{\CC(x)} - \q \circ d$ is almost like evaluation
at~$0$ but is not multiplicative according to
Proposition~\mref{prop:char-intdiffalg} since~$\CC(x)_J$ cannot be an ideal of
the field~$\CC(x)$. In fact, one checks immediately that~$\E(x \cdot 1/x) =
\E(1) = 1$ but~$\E(x) \cdot \E(1/x) = 0 \cdot 0 = 0$.

See Proposition~\mref{pp:rational} for the case when $\D$ here is replaced by
the difference operator or more generally the $\lambda$-difference quotient
operator $\D_\lambda$ with $\lambda\neq 0$ (Example~\mref{ex:diff}).
We refer to~\mcite{Bro} for details on effectively computing the above decomposition
into solvable and transcendental part of rational functions
in the context of symbolic integration algorithms. See also~\mcite{CS} for necessary and sufficient
conditions for the existence of telescopers in
the differential, difference, and $q$-difference case in terms of (generalizations) of residues.

\end{exam}

We can now define what makes a differential algebra such as~$A =
\bfk\diffa{u}$ and~$A = \CC(x)$ adequate for the forthcoming
construction of the free integro-differential algebra.

\begin{defn} Let~$(A,\D)$ be a differential algebra of weight $\lambda$ with derivation~$\D\colon A\to A$.
\begin{enumerate}
  \item If $\lambda=0$, then~$(A, \D)$ is called \emph{regular} if its
  derivation~$\D$ is a regular map. Then a quasi-inverse of $\D$ is called a \emph{quasi-antiderivative}.
\item If $\lambda\neq 0$, then~$(A,\D)$ is called \emph{regular} if its
  derivation~$\D$ is a regular map and the kernel of one of its quasi-inverses
  is a nonunitary $\bfk$-subalgebra of $A$. Such a quasi-inverse of $\D$ is
  called a \emph{quasi-antiderivative}.
\end{enumerate}
      \mlabel{def:reg-diffalg}
\end{defn}

We observe that the class of regular differential algebras is fairly
comprehensive in the zero weight case. It includes all differential algebras
over a field~$\bfk$ since in that case every subspace is complemented, so all
$\bfk$-linear maps are regular. In particular, all differential fields (viewed
as differential algebras over their field of constants) are regular. The
example~$\CC(x)$ is a case in point, but note that Example~\ref{ex:ratfun}
provides an explicit quasi-antiderivative rather than plain existence.

The situation is more complex in the nonzero weight case due to the extra
restriction on the derivation, which we need in our construction of free
integro-differential algebras. If~$\bfk$ is a field, the ring of differential
polynomials~$\bfk\diffa{u}$ is regular for any weight, and we will provide an
explicit quasi-antiderivative that works also when~$\bfk$ is a $\QQ$-algebra but
not a field (Proposition~\mref{pp:diffpoly}).  Moreover, the field of complex
rational functions $\CC(x)$ with its usual difference operator is a regular
differential ring of weight one, and this can be extended to arbitrary nonzero
weight (Proposition~\mref{pp:rational}).

\subsection{Construction of $\efid{A}$}
\mlabel{ss:freec}

According to Theorem~\ref{thm:freex}, the free integro-differential
algebra~$\fid{A}$ can be described by a suitable
quotient. However, for studying this object effectively, a more
explicit construction is preferable. We will achieve this, for a
regular differential algebra~$A$, by defining an
integro-differential algebra~$\efid{A}$, and by showing in the
next subsection that it satisfies the relevant universal
property. Hence we may take~$\efid{A}$ to be~$\fid{A}$.

\subsubsection{Definition of $\efid{A}$ and the statement of
  Theorem~\mref{thm:intdiffa}}

Let $(A, \D)$ be a regular differential algebra with
a fixed quasi-antiderivative~$\q$.

Denote
\begin{equation*}
A_\ci=\im Q\quad\text{and}\quad A_\ct=\ker Q.
\end{equation*}
Then we have the direct sums
\begin{equation*}
  A = A_\ci \oplus \ker{\D}
  \quad\text{and}\quad
  A = \im{\D} \oplus A_\ct
\end{equation*}
with the corresponding projectors~$\E = \id_A - \q \circ \D$ and~$\s = \D \circ
\q$, respectively. As before, we write~$\J = \id_A -\E=\q \circ \D$ and~$\T =
\id_A - \s$ for the complementary projectors. Furthermore, we use the
notation~$K := \ker \D \supseteq \bfk$ in this subsection.

We give now an explicit construction of $\efid{A}$ via tensor
products (all tensors are still over~$\bfk$). First let
\begin{equation*}
  \sha_\ct (A):= \bigoplus_{k\geq 0} A\ot A_\ct^{\otimes k} = A\oplus (A\ot A_\ct) \oplus (A\ot A_\ct^{\ot 2}) +\cdots
\end{equation*}
be the $\bfk$-submodule of $\sha(A)$ in Eq.~(\mref{eq:freerb}). Under our assumption that $A_\ct$ is a subalgebra of $A$ when $\lambda \neq 0$, $\sha_\ct(A)$  is clearly a $\bfk$-subalgebra of $\sha(A)$ under the multiplication in Eqs.~(\mref{eq:shpr0}) and (\mref{eq:shpr}). It is also closed under the derivation $\D_A$ defined in Eq.~(\mref{eq:diffa}).
Alternatively,
$$ \sha_\ct(A)= A\ot \sha^+(A_\ct)$$
is the tensor product algebra where $\sha^+(A_\ct):=\bigoplus\limits_{n\geq 0}
A_\ct^{\ot n}$ is the mixable shuffle algebra~\mcite{Gub,GK1,Ho} on the
$\bfk$-algebra $A_\ct$. In the case~$\lambda=0$, this is the plain shuffle
algebra, where it is sufficient for~$A_\ct$ to have the structure of a~$\bfk$-module.
So a pure tensor $\fraka$ of $A \ot \sha^+(A_\ct)$ is of the
form
\begin{equation} \fraka=a \ot \overline{\fraka}\in A\ot
  A_\ct^{\ot n}\subseteq A^{\ot (n+1)}.  \mlabel{eq:fraka}
\end{equation}
We then define the \emph{length} of $\fraka$ to be $n+1$.

Next let~$\ee\colon A \rightarrow A_\ee$ be an isomorphism
of $K$-algebras, where
\begin{equation}
  A_\ee:=\{\ee(a)\,|\, a\in A\}
  \mlabel{eq:ea}
\end{equation}
denotes a replica of the $K$-algebra $A$, endowed with the zero
derivation. We identify the image~$\ee(K) \subseteq A_\ee$ with~$K$ so that
$\ee(c) = c$ for all~$c \in K$. Finally let
\begin{equation}
  \efid{A}:=A_\ee\ot_K \sha_\ct(A)=A_\ee\ot_K A\ot \sha^+(A_\ct)
  \mlabel{eq:efid}
\end{equation}
denote the tensor product differential algebra of $A_\ee$ and $\sha_\ct(A)$, namely the tensor product algebra where the derivation (again
denoted by $\D_A$) is defined by the Leibniz rule.

\subsubsection{Definition of $\pp_A$}

We will define a linear operator $\pp_A$ on $\efid{A}$. First require that
$\pp_A$ is linear over $A_\ee$. Thus we just need to define $\pp_A(\fraka)$ for
a pure tensor $\fraka$ in $A\ot \sha^+(A_\ct)$. We will accomplish this by
induction on the length $n$ of $\fraka$.  When $n=1$, we have $\fraka = a \in
A$. Then we have
\begin{equation}
  a = \D(\q(a)) + \T(a) \quad\text{with}\quad \T(a) \in A_\ct
  \mlabel{eq:Pu0}
\end{equation}
and we define
\begin{equation}
  \pp_A(a) := \q(a) -\ee(\q(a)) + 1 \ot \T(a).
\mlabel{eq:pa}
\end{equation}
Assume $\pp_A(\fraka)$ has been defined for $\fraka$ of length
$n\geq 1$ and consider the case when $\fraka$ has length $n+1$. Then
$\fraka=a\ot \lbar{\fraka}$ where $a\in A, \lbar{\fraka}\in
A_\ct^{\ot n}$ and we define
\begin{equation}
  \pp_A(a\ot \lbar{\fraka}):= \q(a) \ot \lbar{\fraka} -
  \pp_A(\q(a) \lbar{\fraka})
  -\lambda \, \pp_A(\D(\q(a)) \, \lbar{\fraka})+ 1 \ot \T(a) \ot
  \lbar{\fraka},
  \mlabel{eq:Pu}
\end{equation}
where the first and last terms are manifestly in $A\ot \sha^+(A_\ct)$ while the
middle terms are in $\efid{A}$ by the induction hypothesis. We
write~$\E_A=\id_{\efid{A}}-\pp_A\circ \D_A$ for what will turn out to be the
``evaluation'' corresponding to~$\pp_A$ (see the discussion before
Example~\mref{ex:ratfun}).

We display the following relationship between~$\pp_A$, $\p_A$ and~$\ee$ for
later application.
\begin{lemma}
  \begin{enumerate}
  \item\label{it:ppDD} For $a\in A$, we have $\E_A(a)=\ee(a)$.
  \item\label{it:ppp} For $\lbar\fraka \in \sha^+(A_\ct)$, we have
    $\pp_A(\lbar\fraka)=\p_A(\lbar\fraka)=1\ot \lbar\fraka$.
  \end{enumerate}\mlabel{lem:pf}
\end{lemma}
\begin{proof}
  (\ref{it:ppDD}) Using the direct sum~$A = A_\ci \oplus \ker{\D}$, we
  distinguish two cases. If~$a \in \ker{\D} = K$, then the left-hand side
  is~$a-\pp_A(\D_A(a)) =a- \pp_A(0) = a$; but the right-hand is $a$ as well
  since $\ee\colon A\to A_\ee$ is a $K$-algebra homomorphism. Hence assume~$a \in
  A_\ci = \im{\J}$. In that case~$a = \J(a) = \q(\D(a))$ and hence
  $\T(\D(a))=\D(a)-\D(\q(\D(a)))=0$. So~$\pp_A(\D_A(a)) = \pp_A(\D(a)) = a -
  \ee(a)$ by Eq.~(\mref{eq:pa}).

\noindent
(\ref{it:ppp}) This is a special case of Eqs.~(\mref{eq:Pu0})
    and~(\mref{eq:Pu}) with~$\q(a) = 0$ and~$\T(a) = a$ since~$a \in
    A_\ct$.
\end{proof}

\begin{theorem}
  Let~$(A, \D, \q)$ be a regular differential algebra of
  weight~$\lambda$ with quasi-anti\-derivative~$\q$. Then the triple $(\efid{A},
  \D_A, \pp_A)$, with the natural embedding
  $$i_A\colon
  A \to \efid{A}=A_\ee\ot_K A\ot \sha^+(A_\ct)$$
  to the
  second tensor factor, is the free commutative integro-differential
  algebra of weight $\lambda$ generated by $A$.
  \mlabel{thm:intdiffa}
\end{theorem}

The proof of Theorem~\mref{thm:intdiffa} is given in Section~\mref{ss:intdifpf}.

Since $A_\ct \cong A/\im d$ as $\bfk$-modules, for different choices of $Q$, the
corresponding $A_\ct$ are isomorphic as $\bfk$-modules. Then for $\lambda=0$ the
mixable shuffle (i.e., shuffle) algebras $\sha^+(A_\ct)$ are isomorphic
$\bfk$-algebras since in that case the algebra structure of $A_\ct$ is not used;
see e.g.\@ Section~2.1 of~\cite{GX1}. When $\lambda\neq 0$, for $A_\ct$ from
different choices of $\q$, they are still isomorphic as $\bfk$-modules. But it
is not clear that they are isomorphic as nonunitary
$\bfk$-algebras. Nevertheless, the free commutative integro-differential
algebras derived by Theorem~\mref{thm:intdiffa} are isomorphic due to the
uniqueness of the free objects. See Remark~\mref{rk:subring} for further
discussions.

The following is a preliminary discussion on subalgebras as direct sum factors.

\begin{lemma}
  Let $\T$ and $\s$ be projectors on a unitary $\bfk$-algebra $R$
  such that $\T+\s = \id_{R}$. Then the following statements are
  equivalent:
  \begin{enumerate}
  \item $\im \T=\ker \s$ is a subalgebra;
  \item $\T(\T(x)\T(y))=\T(x)\T(y)$;
  \item $\s(xy)=\s(\s(x)y+x\s(y)-\s(x)\s(y))$.
  \end{enumerate}
\mlabel{lem:projsubalg}
\end{lemma}
\begin{proof}
  ((a) $\Leftrightarrow$ (b)) It is clear since $\T$ is a projector.

\noindent
((a) $\Rightarrow$ (c)) It follows from
\[
\s(\T(x)\T(y))=\s((x-\s(x))(y-\s(y))=0.
\]
\noindent
((c) $\Rightarrow$ (a)) Clearly, the identity implies that $\ker \s$ is a subalgebra.
\end{proof}

If $\s=\D \circ \q$ as above, we obtain from (c) an equivalent identity
\[
\q(xy)=\q(\D(\q(x))y+x\D(\q(y))-\D(\q(x))\D(\q(y)))
\]
in terms of $\q$ and $\D$, since  $\q\circ\D\circ\q=\q$.

\subsection{The proof of Theorem~\mref{thm:intdiffa}}
\mlabel{ss:intdifpf}

We will verify that $(\efid{A}, \D_A, \pp_A)$ is an
integro-differential algebra in Section~\mref{sss:idst} and verify its
universal property in Section~\mref{sss:idfree}.

\subsubsection{The integro-differential algebra structure on $\efid{A}$}
\mlabel{sss:idst}

Since $\D_A$ is clearly a derivation, by
Theorem~\mref{pp:charintdiff}(\mref{it:chb}), we just need to check the two
conditions
\begin{eqnarray}
  \D_A\circ \pp_A &=&\id_{\efid{A}},
  \mlabel{eq:dpid} \\
  \E_A(xy)&=&\E_A(x)\E_A(y), \quad x,y\in \efid{A}.
  \mlabel{eq:dpibp}
\end{eqnarray}
Since $A_\ee$ is in the kernel of $\D_A$ and in the ring of constants for
$\pp_A$, we just need to verify the equations for pure tensors $x=\fraka,
y=\frakb\in A\ot \sha^+(A_\ct)$.

We check Eq.~(\mref{eq:dpid}) by showing $(\D_A\circ
\pp_A)(\fraka)=\fraka$ for $\fraka\in A\ot
\sha^+(A_\ct)$ by induction on the length $n\geq 1$ of
$\fraka$. When $n=1$, we have $\fraka=a\in A$ and obtain
\begin{equation*}
  \D_A(\pp_A(a)) = \D_A(\q(a) -\ee(\q(a)) + 1 \ot \T(a))
  = \D(\q(a)) + \T(a) = a
\end{equation*}
by Eq.~(\mref{eq:Pu0}). Under the induction hypothesis, we consider
$\fraka=a \ot \lbar{\fraka}$ with $\lbar{\fraka}\in A_\ct^{\ot
  n}, n\geq 1$. Then we have
\begin{eqnarray*}
  \D_A(\pp_A(a\ot \lbar{\fraka}))&=&
  \D_A\big(\q(a) \ot \lbar{\fraka} -
  \pp_A(\q(a) \lbar{\fraka})
  -\lambda \, \pp_A(\D(\q(a)) \, \lbar{\fraka})+ 1 \ot \T(a) \ot
  \lbar{\fraka}\big)\\
  &=& \D(\q(a)) \ot \lbar{\fraka} + \q(a) \lbar{\fraka} +
  \lambda \, \D(\q(a)) \lbar{\fraka}
  - \q(a) \lbar{\fraka} - \lambda \, \D(\q(a)) \lbar{\fraka}
  + \T(a) \ot \lbar{\fraka}\\
  &=& \D(\q(a)) \ot \lbar{\fraka} + \T(a) \ot \lbar{\fraka}\\
  &=& a \ot \lbar{\fraka}
\end{eqnarray*}
by Eq.~(\mref{eq:Pu0}) again.

We next verify Eq.~(\mref{eq:dpibp}). If the length of both $x$ and $y$ are one, then $x$ and $y$ are in
$A$. Then by Lemma~\mref{lem:pf}(\mref{it:ppDD}), we have
\begin{equation*}
  E_A(xy)=\ee(xy)=\ee(x)\ee(y)=E_A(x)E_A(y) .
\end{equation*}
If at least one of $x$ or $y$ have length greater than one, then each pure
tensor in the expansion of $xy$ has
length greater than one. Then the equation holds by the following lemma.

\begin{lemma}
For any pure tensor $\fraka=a\ot \lbar{\fraka}\in A\ot \sha^+(A_\ct)$ of length greater than
one we have~$\E_A(\fraka) = 0$.
\mlabel{lem:ezero}
\end{lemma}

\begin{remark}
{\rm
Combining Lemma~\mref{lem:pf}(\mref{it:ppDD}) and Lemma~\mref{lem:ezero} we have $\im \E_A = A_\ee$.
Further, by Eq.~(\mref{eq:kerdimI}), we have
$\ker \D_A = \im \E_A=A_\ee.$
}
\end{remark}

\begin{proof}
For a given $\fraka=a\ot \lbar{\fraka}$ of length greater than one, we compute
{\allowdisplaybreaks
\begin{eqnarray*}
&& \E_A(a\ot \lbar{\fraka}) \\
  &=&a\ot \lbar{\fraka} -\pp_A (\D_A(a\ot
  \lbar{\fraka})) \quad \text{ (by definition of } \E_A)\\
  &=&
  a\ot \lbar{\fraka} - \pp_A(\D(a)\ot
  \lbar{\fraka})-\pp_A(a\lbar{\fraka})-\pp_A (\lambda
  \D(a)\lbar{\fraka}) \quad \text{ (by definition of } \D_A)\\
  &=&
  a\ot \lbar{\fraka} - \q(\D(a)) \ot \lbar{\fraka} +
  \pp_A(\q(\D(a)) \lbar{\fraka})
  +\lambda \, \pp_A(\D(\q(\D(a))) \, \lbar{\fraka})- 1 \ot \T(\D(a)) \ot \lbar{\fraka} \\
&&  -\pp_A(a\lbar{\fraka})-\pp_A (\lambda
  \D(a)\lbar{\fraka}) \quad \text{ (by definition of } \pp_A)\\
  &=& a \ot \lbar\fraka - \q(\D(a)) \ot \lbar\fraka
  + \pp_A(\q(\D(a)) \lbar\fraka) - \pp_A(a \lbar\fraka)
  \quad (\text{by } \D\circ \q\circ \D=\D \text{ and } \T(\D(a))=0)\\
  &=& \E(a) \ot \lbar\fraka - \pp_A( \E(a)
  \lbar\fraka) \quad \text{(by definition of } \E=\id_A - \q\circ \D).
\end{eqnarray*}}
Since $\E(A) = K\subseteq A_\ee$ and $\pp_A$ is taken to be $A_\ee$-linear, from Lemma~\ref{lem:pf}.(\ref{it:ppp}), we obtain
$$ \E_A(a\ot \lbar{\fraka})
=\E(a) (1_A \ot \lbar\fraka - \pp_A(\lbar\fraka))
= 0.$$
\end{proof}

\subsubsection{The universal property}
\mlabel{sss:idfree}
We now verify the universal property of $(\efid{A}, \D_A,
\pp_A)$ as the free integro-differential algebra on $(A,
\D)$: Let $i_A\colon A \to \efid{A}$ be the natural
embedding of $A$ into the the second tensor factor of
$\efid{A}=A_\ee\ot_K A\ot \sha^+(A_\ct)$. Then for any
integro-differential algebra $(R, \DD, \pp)$ and any differential
algebra homomorphism $f\colon (A, \D) \to (R, \DD)$, there is a unique integro-differential algebra homomorphism $\free{f}\colon (\efid{A}, \D_A, \pp_A)\to (R, \DD,
\pp)$ such that $\free{f} \circ i_A = f$. \smallskip

\noindent {\bf The existence of $\free{f}$:} Let a differential
algebra homomorphism $f\colon (A, \D)\to (R, \DD)$ be
given. Note that $f$ is in fact a $K$-algebra homomorphism where the $K$-algebra structure on $R$ is given by $f:K\to R$.
Since $(R,\pp)$ is a commutative Rota-Baxter algebra, by the universal
property of $\sha(A)$ as the free commutative Rota-Baxter algebra
on the commutative algebra $A$, there is a homomorphism
$\tilde{f}\colon (\sha(A),\p_A) \to (R,\pp)$ of commutative
Rota-Baxter algebras such that $\tilde{f}\circ j_A= f$ where $j_A\colon A\to \sha(A)$ is the embedding into the first tensor factor. This means that $\tilde{f}$ is an $A$-algebra homomorphism and, in particular, a $K$-algebra homomorphism. Thus $\tilde{f}$ restricts to a $K$-algebra homomorphism
\begin{equation*}
  \tilde{f}\colon A \ot \sha^+(A_\ct) \to R.
\end{equation*}
Further, $f$ also gives a $K$-algebra homomorphism
\begin{equation*}
  f_\ee\colon A_\ee \to R, \ee(a)\mapsto f(a)-\pp(\DD(f(a))).
\end{equation*}
Thus we get an algebra homomorphism on the tensor product over $K$:
$$\free{f}:=f_\ee \ot_K \tilde{f}\colon A_\ee \ot_K (A \ot \sha^+(A_\ct)) \to R$$
that extends $\tilde{f}$ and $f_\ee$. Further, we have $\free{f}\circ j_A =f.$

It remains to check the equations
\begin{equation}
\free{f}\circ \D_A=\DD\circ \free{f}, \quad
\free{f}\circ \pp_A=\pp\circ \free{f}.
\mlabel{eq:freer}
\end{equation}
Since $A_\ee$ is in the kernel of $\D_A$ and in the ring of
constants of $\pp_A$, we only need to verify the equations when restricted to $A\ot\sha^+(A_\ct)$.

Fix $a\ot \lbar{\fraka}=a(1\ot \lbar{\fraka})\in A\ot \sha^+(A_\ct)$. By
Lemma~\mref{lem:pf}(\mref{it:ppp}), we have
$$\pp(\free{f}(\lbar{\fraka}))=\pp(\tilde{f}(\lbar{\fraka})) =\tilde{f}(\pp_A(\lbar{\fraka}))=\free{f}(1\ot \lbar{\fraka}).$$
Thus we obtain
{\allowdisplaybreaks
\begin{eqnarray*}
  \free{f}(\D_A(a\ot \lbar{\fraka}))&=&
  \free{f}(\D(a)\ot \lbar{\fraka})+\free{f}(a\lbar{\fraka})
  +\free{f}(\lambda \D(a)\lbar{\fraka}) \\
  &=& f(\D(a))\free{f}(1\ot\lbar{\fraka})
  +f(a)\free{f}(\lbar{\fraka}) + \lambda
  f(\D(a))\free{f}(\lbar{\fraka})\\
  &=& \DD(f(a)) \free{f}(1\ot\lbar{\fraka}) + f(a)
  \DD(\pp(\free{f}(\lbar{\fraka}))) + \lambda \DD(f(a))
  \DD(\pp(\free{f}(\lbar{\fraka})))
  \\&=& \DD(f(a)) \free{f}(1\ot\lbar{\fraka}) + f(a) \DD(\free{f}(1\ot
  \lbar{\fraka})) + \lambda \DD(f(a)) \DD(\free{f}(1\ot \lbar{\fraka}))\\
  &=& \DD(f(a)\free{f}(1\ot \lbar{\fraka}))\\
  &=& \DD(\free{f}(a \ot \lbar{\fraka})).
\end{eqnarray*}
}
This proves the first equation in Eq.~(\ref{eq:freer}).
We next prove the second equation by induction on the length $k\geq 1$ of $\fraka:=a\ot \lbar{\fraka}
\in A\ot \sha^+(A_\ct)$.  When $k=1$, we have $\fraka=a\in
A$ and
\begin{eqnarray*}
  \free{f}(\pp_A (a))&=& \free{f} \left( \q(a) - \ee(\q(a))+1\ot \T(a)\right) \\
  &=& f(\q(a)) -f(\q(a))+\pp(\DD(f(\q(a))))+\pp(f(\T(a)))\\
  &=& \pp(f(\D(\q(a))+\T(a)))\\
  &=& \pp(f(a)),
\end{eqnarray*}
using Lemma~\mref{lem:pf}(\mref{it:ppDD}) and~(\mref{it:ppp}). Assume now that
the claim has been proved for $k=n\geq 1$ and consider $\fraka=a\ot
\lbar{\fraka}$ with length $n+1$. Then we have
\begin{eqnarray*}
  \free{f}(\pp_A(a\ot \lbar{\fraka})) &=&
  \free{f}\left(\q(a)\ot\lbar{\fraka}
    -\pp_A(\q(a)\lbar{\fraka}) - \lambda \,
    \pp_A(\D(\q(a))\lbar{\fraka}) +1\ot
    \T(a)\ot\lbar{\fraka}\right)\\
  &=&
  \free{f}(\q(a))\free{f}(\pp_A(\lbar{\fraka})) -
  \free{f}(\pp_A(\q(a)\lbar{\fraka}))\\
&&  -\lambda\free{f}(\pp_A(\D(\q(a))\lbar{\fraka}))
  +\free{f}(\p_A(\T(a)\ot\lbar{\fraka})).
\end{eqnarray*}
Here we have applied Lemma~\mref{lem:pf}(\mref{it:ppp}) in the last
term. Applying the induction hypothesis to the first three terms and using the
fact that the restriction $\tilde{f}$ of $\free{f}$ to $A \ot \sha^+(A_\ct)$
is compatible with the Rota-Baxter
operators in the last term, we obtain
\begin{eqnarray*}
  \free{f}(\pp_A(a\ot \lbar{\fraka})) &=&
  f(\q(a))\pp(\free{f}(\lbar{\fraka}))
  -\pp(\free{f}(\q(a)\lbar{\fraka})) -\lambda
  \pp(\free{f}(\D(\q(a))\lbar{\fraka})) +\pp(\free{f}(\T(a)\ot
  \lbar{\fraka}))\\
  &=& \pp\left (\DD(f(\q(a)))\pp(\free{f}(\lbar{\fraka}))\right) +
  \pp\left (f(\T(a))\free{f}(\p_A(\lbar{\fraka}))\right),
\end{eqnarray*}
where we have used integration by parts in
Theorem~\mref{pp:charintdiff}(\ref{it:intpart}) in the last step.
On the other hand, we have
\begin{eqnarray*}
  \pp(\free{f}(a\ot \lbar{\fraka}))&=&
  \pp(f(a)\free{f}(\p_A(\lbar{\fraka})))
  \\&=&
  \pp\left (f(\D(\q(a))+\T(a))\free{f}(\p_A(\lbar{\fraka}))\right)
  \\&=&
  \pp\left (\DD(f(\q(a)))\pp(\free{f}(\lbar{\fraka}))\right) +
  \pp\left (f(\T(a))\free{f}(\p_A(\lbar{\fraka}))\right).
\end{eqnarray*}
Thus we have completed the proof of
the existence of the integro-differential algebra homomorphism
$\free{f}$.  \smallskip

\noindent
{\bf The uniqueness of $\free{f}$:}
Suppose $\free{f}_1\colon \efid{A} \to R$ is a homomorphism of
integro-differential algebras such that $\free{f}_1\circ i_A=f$.
For $1\ot a_1\ot \cdots\ot a_n \in \sha^+(A_\ct)$, we have
\begin{eqnarray*}
  \free{f}_1(1\ot a_1\ot \cdots \ot a_n)&=&
  \free{f}_1\left (\pp_A (a_1\pp_A (\cdots \pp_A
    (a_n)\cdots ))\right) \\
  &=&
  \pp(f(a_1)\pp(\cdots \pp(f(a_n))\cdots)) \\
  &=&
  \free{f} \left (\pp_A (a_1\pp_A( \cdots \pp_A (a_n)\cdots))\right)\\
  &=&
  \free{f}(1\ot a_1\ot\cdots\ot a_n).
\end{eqnarray*}
Thus the restrictions of $\free{f}$ and $\free{f}_1$ to $A \ot \sha^+(A_\ct)$
are the same. Further, by Lemma~\mref{lem:pf}(\mref{it:ppDD}),
\begin{equation*}
  \free{f}_1(\ee(a)) =f(a)-\free{f}_1(\pp_A(\D_A(a)))
  =f(a)-\pp(\DD(f(a))=\free{f}(\ee(a)).
\end{equation*}
Hence the restrictions of $\free{f}$ and $\free{f}_1$
to $A_\ee$ are also the same. As these restrictions to $A\ot \sha^+(A_\ct)$ and $A_\ee$ are $K$-homomorphisms, by the universal property of the tensor product over $K$, $\free{f}$ and $\free{f}_1$ agree on
$\efid{A}=A_\ee\ot_K A\ot \sha^+(A_\ct).$
This proves the uniqueness of $\free{f}$ and thus completes the proof of Theorem~\mref{thm:intdiffa}.

\subsection{Examples of regular differential algebras}
\mlabel{ss:example}
In this section we show that some common examples of differential algebras, namely the algebra of differential polynomials and the algebra of rational functions, are regular where the weight can be taken arbitrary.

\subsubsection{Rings of differential polynomials}
\mlabel{sss:diffpoly}

Our main goal in this subsection is to prove
that $(\bfk\diffa{u},\D)$ is a regular differential algebra for any weight, and to give an
explicit quasi-antiderivative~$\q$ for~$\D$.

We start by introducing some definitions for classifying the elements
of~$A=\bfk \diffa{u}$. Let $u_i, i\geq 0,$ be the $i$-th derivation of
$u$. Then $\bfk\diffa{u}$ is the polynomial algebra on $\{u_i\ |\ i\geq
0\}.$ For $\alpha=(\alpha_0,\ldots,\alpha_k)\in \NN^{k+1}$, we write
$u^{\alpha}=u_0^{\alpha_0}\cdots u_k^{\alpha_k}$. Furthermore, we use the
convention that~$u^{\alpha} = 1$ when~$\alpha \in \NN^0$ is the
degenerate tuple of length zero. Then all monomials of $\bfk\diffa{u}$
are of the form $u^{\alpha}$, where~$\alpha$ contains no trailing
zero. The \emph{order} of such a monomial~$u^{(\alpha_0, \dots, \alpha_k)} \ne
1$ is defined to be $k$; the order of~$u^{()} = 1$ is set to~$-1$. The
order of a nonzero differential polynomial is defined as the maximum
of the orders of its monomials. The following classification of
monomials is crucial~\cite{GD,Bi}: A monomial~$u^{\alpha}$ of
order~$k$ is called \emph{functional} if either $k \le 0$ or $\alpha_k >
1$. We write
\begin{equation*}
  A_T=\bfk\{ u^{\alpha} \mid  u^\alpha \text{ is functional}\}
\end{equation*}
for the corresponding submodule. Since the product of two functional monomials is again functional, $A_T$ is in fact a $\bfk$-subalgebra of $A$.
Furthermore, we
write~$A_J$ for the submodule generated by all monomials
$u^{\alpha} \ne 1$.

\begin{prop}
  \label{pp:diffpoly}
  For any $\lambda\in \bfk$, the canonical derivation~$\D\colon A \to A$ of weight $\lambda$ defined in Theorem~\mref{thm:freediff} admits a
  quasi-antiderivative~$\q$ with associated direct sums~$A =
  A_T \oplus \im{\D}$ and~$A = A_J \oplus \ker{\D}$.
\end{prop}
\begin{proof}
  The main work goes into showing the direct sum~$A = A_T \oplus \im{\D}$. We
  first show $A_T \cap \im \D=0$. Let $x\in A$.  If~$x$ has order~$-1$, it is an
  element of~$\bfk$ so that~$d(x) = 0$. If~$x$ has order~$k \ge 0$, we
  distinguish the two cases of $\lambda=0$ and $\lambda\neq 0$. If $\lambda=0$,
  then we have~$\D(x) = (\partial x/\partial u_k) \, u_{k+1} + \tilde{x}$, where
  all terms of~$\tilde{x}$ have order at most~$k$.  Hence $\D(x)\not\in A_T$ and
  therefore we have $A_T\cap \im \D=0$.

  We now turn to the case when $\lambda\neq 0$. By Eq.~(\mref{eq:diff}) and an inductive argument, we find that for a product $w=\prod_{i\in I} w_i$ in $A$, we have
   $$ \D(w)= \sum_{\varnothing\neq J\subseteq I} \lambda^{|J|-1}\prod_{i\in J}\D(w_i)\prod_{i\not\in J} w_i.$$
Then for a given monomial $u^\alpha=u^{(\alpha_0,\ldots,\alpha_k)}=\prod_{i=0}^k u_i^{\alpha_i}$ of order $k$ we have
\begin{eqnarray}
\D(u^\alpha)&=& \sum_{0\leq\beta_i\leq \alpha_i, \sum_{i=0}^k \beta_i\geq 1} \lambda^{\beta_0
+ \cdots + \beta_k - 1}
\prod_{i=0}^k \binc{\alpha_i}{\beta_i} u_i^{\alpha_i-\beta_i} u_{i+1}^{\beta_i}\notag\\
&=& \sum_{0\leq\beta_i\leq \alpha_i, \sum_{i=0}^k \beta_i\geq 1} \lambda^{\beta_0
+ \cdots + \beta_k - 1} \left(\prod_{i=0}^k \binc{\alpha_i}{\beta_i} u_i^{\alpha_i-\beta_i+\beta_{i-1}} \right) u_{k+1}^{\beta_k},
\mlabel{eq:dorder}
\end{eqnarray}
with the convention $\beta_{-1}=0$.
Consider the reverse lexicographic order on monomials of order $k+1$:
$$ (\beta_0,\dots,\beta_{k+1})<(\gamma_0,\dots,\gamma_{k+1}) \Leftrightarrow \exists\, 0\leq n\leq k+1 \;(\beta_i=\gamma_i \text{ for } n< i\leq k+1\ \text{ and }\ \beta_n<\gamma_n).$$
The smallest monomial {\em of order $k+1$} under this order in the sum in Eq.~(\mref{eq:dorder}) is given by  $u_0^{\alpha_0}\cdots
u_{k-1}^{\alpha_{k-1}}u_k^{\alpha_k-1}u_{k+1}$ when $\beta_k=1$ and
$\beta_0=\cdots = \beta_{k-1}=0$, coming from $u_0^{\alpha_0}\cdots
u_{k-1}^{\alpha_{k-1}}\D(u_k^{\alpha_k})$. Thus for two monomials of order~$k$
with $u^\alpha<u^\beta$ under this order, the least monomial of order $k+1$ in
$\D(u^\alpha)$ is smaller than the least monomial of order $k+1$ in
$\D(u^\beta)$. In particular, for the least monomial $u^{\alpha}$ of order $k$
of our given element $x$ of order $k\geq 0$, the least monomial of order $k+1$
in $\D(u^\alpha)$ is the least monomial of order $k+1$ in $\D(x)$ and is given
by $u_0^{\alpha_0}\cdots u_{k-1}^{\alpha_{k-1}}u_k^{\alpha_k-1}u_{k+1}$. Since
this monomial is not in $A_\ct$, it follows that $\D(x)$ is not in $A_\ct$,
showing that $A_T\cap \im \D=0$.

Note that the previous argument shows in particular that $\D(x)\neq 0$ for $x\not\in \bfk$. Thus we have
$$ A = A_J \oplus \bfk.$$

We next show that every monomial $u^{\alpha}$ in $\bfk\diffa{u}$ is in $A_T+\im
\D$. We prove this by induction on the order of~$u^{\alpha}$. If the order
is~$-1$ or~$0$, then~$u^{\alpha} \in A_T$ by definition. Assuming the claim
holds for differential monomials of order less than~$k > 0$, consider now a
monomial~$u^{\alpha}$ of order~$k$ so
that~$\alpha=(\alpha_0,\dots,\alpha_k)$. If $u^{\alpha} \in A_T$, we are
done. If not, we must have $\alpha_k=1$. Then we distinguish the cases when
$\lambda=0$ and $\lambda\neq 0$.  If $\lambda=0$, then
\begin{equation*}
  \begin{aligned}
    u^{\alpha}&=u_0^{\alpha_0}\cdots u_{k-1}^{\alpha_{k-1}} u_k\\
    &= u_0^{\alpha_0}\cdots u_{k-2}^{\alpha_{k-2}} \, \tfrac{1}{\alpha_{k-1}+1} \, \D(u_{k-1}^{\alpha_{k-1}+1})\\
    &= \D(u_0^{\alpha_0}\cdots u_{k-2}^{\alpha_{k-2}} \,
    \tfrac{1}{\alpha_{k-1}+1} \, u_{k-1}^{\alpha_{k-1}+1})-
    \D(u_0^{\alpha_0}\cdots u_{k-2}^{\alpha_{k-2}}) \tfrac{1}{\alpha_{k-1}+1} \,
    u_{k-1}^{\alpha_{k-1}+1}.
  \end{aligned}
\end{equation*}
Now the first term in the result is in $\im \D$ and the second term is in
$A_T+\im \D$ by the induction hypothesis, allowing us to complete the induction
when $\lambda=0$.

Now consider the case when $\lambda\neq 0$. Suppose the claim does not hold for
some monomials $u^\alpha=u^{(\alpha_0,\cdots,\alpha_{k-1},1)}$ of order $k$. Among these monomials, there is one such
that the exponent vector $\alpha=(\alpha_0,\dots,\alpha_{k-1},1)$ is minimal with
respect to the lexicographic order:
$$ (\alpha_0,\dots,\alpha_{k-1},1)<(\beta_0,\dots,\beta_{k-1},1) \Leftrightarrow \exists\, 0\leq n\leq k-1 \; (\alpha_i=\beta_i \text{ for } 1\leq i< n\ \text{ and }\  \alpha_n<\beta_n).$$
By Eq.~\eqref{eq:dorder}, we have
{\allowdisplaybreaks
\begin{eqnarray*}
d(u_{k-1}^{\alpha_{k-1}+1}) &=&\sum_{\beta_{k-1}=1}^{\alpha_{k-1}+1} \binc{\alpha_{k-1}+1}{\beta_{k-1}} \lambda^{\beta_{k-1}-1} u_{k-1}^{\alpha_{k-1}+1-\beta_{k-1}}u_k^{\beta_{k-1}} \\ &=&(\alpha_{k-1}+1) u_{k-1}^{\alpha_{k-1}}u_k  +\sum_{\beta_{k-1}=2}^{\alpha_{k-1}+1} \binc{\alpha_{k-1}+1}{\beta_{k-1}} \lambda^{\beta_{k-1}-1} u_{k-1}^{\alpha_{k-1}+1-\beta_{k-1}}u_k^{\beta_{k-1}}.
\end{eqnarray*}
}
So
$$u_{k-1}^{\alpha_{k-1}}u_k=
\tfrac{1}{\alpha_{k-1}+1} d(u_{k-1}^{\alpha_{k-1}+1})
-\sum_{\beta_{k-1}=2}^{\alpha_{k-1}+1} \tfrac{\lambda^{\beta_{k-1}-1}}{\alpha_{k-1}+1} \binc{\alpha_{k-1}+1}{\beta_{k-1}} u_{k-1}^{\alpha_{k-1}+1-\beta_{k-1}}u_k^{\beta_{k-1}}.$$
Thus
\begin{eqnarray*}
u^\alpha&=& u_0^{\alpha_0}\cdots u_{k-1}^{\alpha_{k-1}}u_k \\
&=& u_0^{\alpha_0}\cdots u_{k-2}^{\alpha_{k-2}}
\tfrac{1}{\alpha_{k-1}+1} d(u_{k-1}^{\alpha_{k-1}+1})
\\
&&-\sum_{\beta_{k-1}=2}^{\alpha_{k-1}+1} \tfrac{\lambda^{\beta_{k-1}-1}}{\alpha_{k-1}+1} \binc{\alpha_{k-1}+1}{\beta_{k-1}} u_0^{\alpha_0}\cdots u_{k-2}^{\alpha_{k-2}} u_{k-1}^{\alpha_{k-1}+1-\beta_{k-1}}u_k^{\beta_{k-1}}.
\end{eqnarray*}
The monomials in the sum are in $A_\ct$. For the first term, by Eq.~\eqref{eq:diff}, we have
  \begin{equation*}
    \begin{aligned}
      & u_0^{\alpha_0}\cdots u_{k-2}^{\alpha_{k-2}} \, \tfrac{1}{\alpha_{k-1}+1} \, \D(u_{k-1}^{\alpha_{k-1}+1})\\
      &= \D(u_0^{\alpha_0}\cdots u_{k-2}^{\alpha_{k-2}} \, \tfrac{1}{\alpha_{k-1}+1} \,
      u_{k-1}^{\alpha_{k-1}+1})- \D(u_0^{\alpha_0}\cdots u_{k-2}^{\alpha_{k-2}})
      \tfrac{1}{\alpha_{k-1}+1} \, u_{k-1}^{\alpha_{k-1}+1}\\
      &\quad - \lambda \,\D (u_0^{\alpha_0}\cdots u_{k-2}^{\alpha_{k-2}}) \D(\tfrac{1}{\alpha_{k-1}+1} \,u_{k-1}^{\alpha_{k-1}+1}).
    \end{aligned}
  \end{equation*}
As in the case of $\lambda=0$, the first term in the result is in $\im \D$ and the second term has the desired decomposition by the induction hypothesis. Applying Eq.~(\mref{eq:dorder}) to both derivations in the third term, we see that the term is a linear combination of monomials of the form $u^\gamma=u^{(\gamma_0,\cdots,\gamma_k)}$ where
$$\gamma= (\alpha_0-\beta_0,\alpha_1-\beta_1+\beta_0,\dots, \alpha_{k-2}-\beta_{k-2}+\beta_{k-3}, \alpha_{k-1}+1-\beta_{k-1}+\beta_{k-2},\beta_{k-1})$$
for some $0\leq \beta_i\leq \alpha_i, 0\leq i\leq k-2$ with $\sum\limits_{i=0}^{k-2}\beta_i\geq 1$ and $\beta_{k-1}\geq 1$. If such a monomials has $\beta_{k-1}\geq 2$, then the monomial is already in $A_\ct$. If such a monomial has $\beta_{k-1}=1$, then it has order $k$ and has lexicographic order less than $u^\alpha$ since $\sum\limits_{i=0}^{k-2}\beta_i\geq 1$. By the minimality of $u^\alpha$, this monomial is in $A_T+\im \D $. Hence $u^\alpha$ is in $A_T+\im \D$. This is a contradiction, allowing us to completes the induction when $\lambda\neq 0$.

With the two direct sum decompositions, the quasi-antiderivative $Q$ is obtained by Proposition~\mref{pp:reg}.
\end{proof}

We can thus conclude that~$\bfk\diffa{u}$ is indeed a regular
differential algebra, as claimed earlier. Hence the
construction~$\efid{u} = \efid{\bfk\diffa{u}}$ developed in
Section~\ref{ss:freec} does yield the free integro-differential
algebra over the single generator~$u$.

\begin{prop}
  Let $\bfk$ be a commutative $\QQ$-algebra. Then the free integro-differential algebra~$\fid{\bfk\diffa{u}}$ is a
  polynomial algebra.
\end{prop}
\begin{proof}
  We first take the coefficient ring to be $\QQ$.  Since~$\fid{\QQ\diffa{u}}$ is
  isomorphic to~$\efid{\QQ\diffa{u}}$, which is given by Eq.~\eqref{eq:efid}
  with~$A = \QQ\diffa{u}$, it suffices to ensure that~$\sha^+(A_T)$ is a
  polynomial algebra. Now observe that~$A_T = \QQ F$ is the monoid algebra
  generated over the set~$F$ of functional monomials. One checks immediately
  that the functional monomials~$F$ form a monoid under multiplication. Hence
  Theorem~2.3 of~\cite{GX1} is applicable, and we see that the mixable shuffle
  algebra~$\sha^+(A_T) = \operatorname{MS}_{\QQ,\lambda}(F)$ is isomorphic
  to~$\QQ[\operatorname{Lyn}(F)]$, where~$\operatorname{Lyn}(F)$ denotes the set
  of Lyndon words over~$F$. This proves the proposition when $\bfk = \QQ$.  Then
  the conclusion follows for any commutative $\QQ$-algebra $\bfk$ since
  $\efid{\bfk\diffa{u}}\cong\bfk \ot_\QQ \efid{\QQ\diffa{u}}.$
\end{proof}

\subsubsection{Rational functions}
\mlabel{sss:ratfunc}
We show that the algebra of rational functions with derivation of any weight is regular.

\begin{prop}
Let $A=\CC(x)$. For any $\lambda\in \CC$ let
\begin{equation}
\D_\lambda \colon A\to A, f(x)\mapsto \left\{\begin{array}{ll}
\frac{f(x+\lambda)-f(x)}{\lambda}, &\lambda\neq 0, \\
f'(x), & \lambda =0,
\end{array} \right .
\mlabel{eq:ratder}
\end{equation}
be the $\lambda$-derivation introduced in
Example~\mref{ex:diff}$($\mref{it:divdiff}$)$. Then~$\D_\lambda$ is regular. In
particular the difference operator on $\CC(x)$ is a regular derivation of weight
one.  \mlabel{pp:rational}
\end{prop}

\begin{proof}
We have considered the case of $\lambda=0$ in Example~\mref{ex:ratfun}. Modifying the notations there, any rational function can be uniquely expressed as
  \begin{equation}
    r + \sum_{i=1}^k \sum_{j=1}^{n_i}
    \frac{\gamma_{ij}}{(x-\alpha_{ij})^i},
    \mlabel{eq:pfdecomp1}
  \end{equation}
where $r\in\CC[x], \alpha_{ij}\in \CC$ are distinct for any given $i$ and $\gamma_{ij}\in \CC$ are nonzero. Let $0\neq \lambda\in \CC$ be given. We have the direct sum of linear spaces
$$
\CC[x] \oplus \calr= \CC[x] \oplus \bigoplus_{i\geq 1} \calr_i,
$$
where $\calr$ is the linear space from the fractions in Eq.~(\mref{eq:pfdecomp1}), namely the linear space with basis $1/(x-\alpha)^i, \alpha\in \CC, 1 \leq i$, and $\calr_i$, for fixed $i\geq 1$, is the linear subspace with basis $1/(x-\alpha)^i, \alpha\in \CC$.

We note that the \emph{$\lambda$-divided falling factorials}

$$ \binom{x}{n}_{\!\lambda}:=\frac{x(x-\lambda)(x-2\lambda)\cdots (x-(n+1)\lambda)}{n!},\quad n\geq 0,$$
with the convention $\binom{x}{0}_{\!\lambda}=1$, form a $\CC$-basis of $\CC[x]$.
In fact,
$$ \binom{x}{n}_{\!\lambda} \! = \frac{1}{n!} \sum_{k=0}^n s(n,k) \lambda^{n-k}x^k, \quad
x^n = n! \sum_{k=0}^n S(n,k)\lambda^{n-k} \binom{x}{n}_{\!\lambda}, \quad  n\geq 0,
$$
where $s(n,k)$ and $S(n,k)$ are Stirling numbers of the first and second kind,
respectively; see~\mcite{Gub,Gst} for example. By a direct computation, we have
$$\D_\lambda \left(\binom{x}{n}_{\!\lambda} \, \right)=\frac{\binom{x+\lambda}{n}_{\!\lambda} - \binom{x}{n}_{\!\lambda} }{\lambda} = \binom{x}{n-1}_{\!\lambda}.$$
Thus $\D_\lambda(\CC[x])=\CC[x]$ and hence $\CC[x]\subseteq \im \D_\lambda$.
We next note that $\calr$, as well as $\calr_k$, is also closed under the operator $\D_\lambda$ since
$$ \lambda \, \D_\lambda \left (\sum_{i=1}^k \sum_{j=1}^{n_i}
    \frac{\gamma_{ij}}{(x-\alpha_{ij})^i}\right)
    = \sum_{i=1}^k \sum_{j=1}^{n_i}
    \frac{\gamma_{ij}}{(x-(\alpha_{ij}-\lambda))^i} -
    \sum_{i=1}^k \sum_{j=1}^{n_i}
    \frac{\gamma_{ij}}{(x-\alpha_{ij})^i}.$$
Further, for any $n \geq 0$ and $f(x)\in\CC(x)$, we have
$$
\lambda \, \D_\lambda\left(\sum_{i=0}^nf(x+i\lambda)\right) =f(x+(n+1)\lambda)-f(x),
$$
and similarly for $n<0$,
\[
\lambda \, \D_\lambda\left(\sum_{i=n}^{-1} f(x+i\lambda)\right) =f(x)-f(x+n\lambda),
\]
Thus for any $n\in \ZZ$, we have
$$ f(x) \equiv f(x+n\lambda) \mod \im \D_\lambda.$$
In particular,
$$ 1/(x-\alpha)^i \equiv 1/(x-(\alpha-n\lambda))^i \mod \im \D_\lambda$$
and hence
$$ 1/(x-\alpha)^i \equiv 1/(x-\beta)^i \mod \im \D_\lambda,$$
for some $\beta\in \CC$ with the real part $\re (\beta)\in [0,\abs{\re
  (\lambda)}).$
Consequently, any fraction in $\calr$ is congruent modulo $\im \D_\lambda$ to an element of
\begin{equation} \Cxto := \left \{ \sum_{i=1}^k \sum_{j=1}^{n_i}
    \frac{\gamma_{ij}}{(x-\alpha_{ij})^i}\in \calr \Big| \re(\alpha_{ij})\in [0,\abs{\re (\lambda)}) \right\}.
    \mlabel{eq:resfract}
\end{equation}
That is,
$$\CC(x)= \im\D_\lambda + \Cxto.$$

On the other hand, suppose there is a nonzero function
$$f(x)=\sum\limits_{i=1}^k \sum\limits_{j=1}^{n_i}
    \frac{\gamma_{ij}}{(x-\alpha_{ij})^i} \in \im\D_\lambda\cap \Cxto.$$
Thus there is $g(x)=\sum\limits_{i=1}^k \sum\limits_{j=1}^{m_i}
    \frac{\gamma_{ij}}{(x-\beta_{ij})^i}$ such that $\D_\lambda(g(x))=f(x).$
The range of $i$ in $f(x)$ and $g(x)$ are the same since $\D_\lambda (\calr_i)\subseteq \calr_i$.
Let $f(x)=\sum\limits_{i=1}^k f_i(x)$ and $g(x)=\sum\limits_{i=1}^kg_i(x)$ be the homogeneous decompositions of $f$ and $g$. Then $\D_\lambda (g_i(x))=f_i(x), 1\leq i\leq k$. Fix $1\leq i\leq k$ and take $\re (\lambda)>0$ for now. List
$\beta_{i,1}<\cdots<\beta_{i,m_i}$ according to their lexicographic order from the pairs $(a,b)\leftrightarrow a+i\,b\in \CC$. Then we have
$$ \lambda \, \D_\lambda (g_i(x))= \sum_{j=1}^{m_i} \frac{\gamma_{ij}}{(x-(\beta_{ij}-\lambda))^i}-\sum_{j=1}^{m_i} \frac{\gamma_{ij}}{(x-\beta_{ij})^i}.$$
The first fraction in the first sum, $1/(x-(\beta_{i,1}-\lambda))^i$, is not the
same as any other fraction in the first sum since they are translations by
$\lambda$ of distinct fractions in $f_i$, and is not the same as any fraction in
the second sum since $\re(\beta_{i,1}-\lambda)<\re(\beta_{i,1}) \leq
\re(\beta_{ij})$ for $1\leq j\leq m_i.$ Similarly the last fraction in the
second sum, $1/(x-\beta_{i,m_i})^i$, is not the same as any other terms in the
sums. Thus they both have nonzero coefficients in $\D_\lambda(g_i(x))$. But
$$ \re(\beta_{i,m_i})-\re(\beta_{i,1}-\lambda) =\re(\beta_{i,m_i}-(\beta_{i,1}-\lambda)) =\re(\beta_{i,m_i}-\beta_{i,1})+\re(\lambda)\geq \re(\lambda).$$
Hence $\re(\beta_{i,m_i})$ and $\re(\beta_{i,1}-\lambda)$ cannot both be in $[0,\re (\lambda))$. Thus $\D_\lambda(g_i)$ and hence $\D_\lambda(g)$ cannot be in $\Cxto$. This is a contradiction, showing that $\im \D_\lambda\cap \Cxto=0$. When $\re(\lambda)<0$, we get analogously $\im \D_\lambda\cap \Cxto=0$. Thus we have proved
\begin{equation}
\CC(x)=\im \D_\lambda \oplus \Cxto.
\mlabel{eq:fracdecomp1}
\end{equation}
Note that $\Cxto$ is closed under multiplication, hence is a nonunitary subalgebra of $\CC(x)$.

The above argument shows that $\D_\lambda(g)$ is in $\Cxto$ for $g\in \calr$ only when $g=0$. Thus $\ker \D_\lambda \cap \calr=0$. Since $\D_\lambda$ preserves the decomposition $\CC(x)=\CC[x]\oplus \calr$, we have $\ker \D_\lambda = \ker (\D_\lambda)\big|_{\CC[x]}=\CC.$
Thus we have the direct sum decomposition
\begin{equation}
 \CC(x)=\ker \D_\lambda \oplus (x\CC[x]\oplus \calr),
 \mlabel{eq:fracdecomp2}
 \end{equation}
and hence $\D_\lambda$ is injective on
$x\CC[x]\oplus \calr$ with image $\im \D_\lambda$. Therefore $\D_\lambda$ is regular with quasi-antiderivative $Q$ defined to be the inverse of
$$\D_\lambda: x\CC[x]\oplus \calr \to \im \D_\lambda$$
on $\im \D_\lambda$ and to be zero on its complement $\Cxto$; see
Proposition~\mref{pp:reg}.
\end{proof}

\begin{remark}
{\rm We remark that the subalgebra of $\CC(x)$ that is a complement of $\im
  \D_\lambda$ is not unique, thus giving different quasi-antiderivatives. In fact, from the proof of Proposition~\mref{pp:rational} it is apparent that in the decomposition (\mref{eq:fracdecomp1}) one can replace $\Cxto$ by
$$ \Cxta
=\left \{ \sum_{i=1}^k \sum_{j=1}^{n_i}
    \frac{\gamma_{ij}}{(x-\alpha_{ij})^i}\in \calr\,\Big|\, \re(\alpha_i)\in [a,a+|\re(\lambda)|) \right\},
$$
for any given $a\in \RR$. These two subalgebras are isomorphic since $\Cxta$ is isomorphic to the polynomial $\CC$-algebra with generating set
$$ \left\{ \frac{1}{x-\alpha}\,\Big|\, \alpha\in [a,a+|\re(\lambda)|)\right\}.$$
}
\mlabel{rk:subring}
\end{remark}

\begin{remark}
\label{s:concl}
In conclusion, we have given the first construction for the free
integro-differential algebra~$\efid{A}$ over a given regular differential
algebra~$A$. In several ways, this construction is similar to the
integro-differential polynomials of~\cite{RR1,RRTB1}. This will be clear when
one writes out the elements~$a_0 \ot a_1 \ot a_2 \ot \cdots$ of
Eq.~\eqref{eq:fraka} in the form~$a_0 \cum a_1 \cum a_2 \cum \cdots$. 
But there are
also some important differences:
\begin{enumerate}
\item\label{it:intdiffpol} The integro-differential polynomials are the
  polynomial algebra in the variety of integro-differential algebras of weight zero, not the
  free algebra in this category. In fact, the polynomial algebra is always a
   free product of the coefficient algebra and the free algebra by Theorem~4.31
  of~\cite{LN}.
\item The construction of~\cite{RR1} uses the language of term
  algebras and rewrite systems whereas in this paper we use a more
  abstract approach through tensor products.
\item In the integro-differential polynomials, the starting point is a given
  integro-differential algebra~$(A, \DD, \pp)$ instead of a regular differential
  algebra as in the present paper. In the former case we can construct nested
  integrals over differential polynomials with coefficients in~$\bfk[x]$,
  whereas in the latter case we can only treat differential polynomials with
  trivial coefficients (i.e.\@ the derivation vanishes on them).
\end{enumerate}
It would be interesting to apply the methods used in this paper to rederive and
generalize the construction of the integro-differential polynomials
of~\cite{RR1}. This would also shed some light on the constructive meaning of
the free union mentioned in Item~(\mref{it:intdiffpol}) above. An important step
in this direction might be generalizing Section~\ref{sss:diffpoly} to
differential polynomials with nonzero derivation on the coefficient ring~$\bfk$. See~\cite{GG} for a construction of the free integro-differential algebra on one generator by the method of Gr\"obner-Shirshov basis.
\end{remark}

\medskip

\noindent
\emph{Acknowledgements}:
L.~Guo acknowledges support from NSF grant DMS 1001855. G.~Regensburger was supported by the Austrian Science Fund (FWF): J 3030-N18. M.~Rosenkranz acknowledges support from the EPSRC First Grant EP/I037474/1.

\end{document}